\documentclass[a4paper, 12pt]{amsart}

\usepackage{amsmath}
\usepackage{amssymb}
\usepackage{amsthm}
\usepackage{graphicx}
\usepackage[dvipdfmx]{hyperref}

\makeatletter

\@addtoreset{equation}{section}
\makeatother 

\theoremstyle{plain}
\newtheorem{thm}{Theorem}[section]
\newtheorem*{thm*}{Theorem}
\newtheorem{prop}[thm]{Proposition}
\newtheorem{lem}[thm]{Lemma}
\newtheorem{cor}[thm]{Corollary}

\theoremstyle{definition}
\newtheorem{ex}[thm]{Example}
\newtheorem{defi}[thm]{Definition}

\theoremstyle{remark}
\newtheorem{rem}[thm]{Remark}

\renewcommand{\epsilon}{\varepsilon}

\newcommand{\ric}{\operatorname{Ric}}

\newcommand{\Hess}{\operatorname{Hess}}

\newcommand{\Lip}{\mathrm{Lip}}

\newcommand{\inte}{\mathrm{Int}\,}
\newcommand{\cut}{\mathrm{Cut}\,}

\newcommand{\bm}{\partial M}
\newcommand{\op}{\operatorname{op}}

\newcommand{\Var}{\operatorname{Var}}

\newcommand{\IR}{\operatorname{InRad}}
\newcommand{\SR}{\operatorname{SpecRad}}

\newcommand{\rarrow}{\rightarrow}
\newcommand{\orarrow}{\overrightarrow}

\newcommand{\olarrow}{\overleftarrow}

\usepackage{latexsym}

\title[Geometric and spectral properties of directed graphs]{Geometric and spectral properties of directed graphs under a lower Ricci curvature bound}

\author{Ryunosuke Ozawa}
\address{Advanced Institute for Materials Research (AIMR), Tohoku University, 2-1-1 Katahira, Aoba-ku, Sendai, 980-8577, Japan}
\email{ryunosuke.ozawa.b8@tohoku.ac.jp}

\author{Yohei Sakurai}
\address{Advanced Institute for Materials Research (AIMR), Tohoku University, 2-1-1 Katahira, Aoba-ku, Sendai, 980-8577, Japan}
\email{yohei.sakurai.e2@tohoku.ac.jp}

\author{Taiki Yamada}
\address{Research Institute for Humanity and Nature, 457-4 Motoyama, Kamigamo, Kita-ku, Kyoto, 603-8047, Japan}
\email{t.yamada@chikyu.ac.jp}

\date{June 21, 2020}

\begin{document}
\maketitle

\begin{abstract}
For undirected graphs,
the Ricci curvature introduced by Lin-Lu-Yau has been widely studied from various perspectives,
especially geometric analysis.
In the present paper,
we discuss generalization problem of their Ricci curvature for directed graphs.
We introduce a new generalization for strongly connected directed graphs by using the mean transition probability kernel
which appears in the formulation of the Chung Laplacian.
We conclude several geometric and spectral properties under a lower Ricci curvature bound extending previous results in the undirected case.
\end{abstract}

\section{Introduction}

Ricci curvature is one of the most fundamental objects in Riemannian geometry.
Based on a geometric observation on (smooth) Riemannian manifolds,
Ollivier \cite{O} has introduced the coarse Ricci curvature for (non-smooth) metric spaces by means of the Wasserstein distance which is an essential tool in optimal transport theory.
Modifying the formulation in \cite{O},
Lin-Lu-Yau \cite{LLY} have defined the Ricci curvature for undirected graphs.
It is well-known that
a lower Ricci curvature bound of Lin-Lu-Yau \cite{LLY} implies various geometric and analytic properties (see e.g., \cite{BJL}, \cite{BRT}, \cite{JL}, \cite{LLY}, \cite{MW}, \cite{P}, and so on).

There have been some attempts to generalize the Ricci curvature of Lin-Lu-Yau \cite{LLY} for directed graphs.
The third author \cite{Y1} has firstly proposed a generalization of their Ricci curvature (see Remark \ref{rem:various directed Ricci curvature} for its precise definition).
He computed it for some concrete examples, and given several estimates.
Eidi-Jost \cite{EJ} have recently introduced another formulation (see Remark \ref{rem:various directed Ricci curvature}).
They have applied it to the study of directed hypergraphs.

We are now concerned with the following question:
What is the suitable generalization of the Ricci curvature of Lin-Lu-Yau \cite{LLY} for strongly connected directed graphs?
In this paper,
we provide a new Ricci curvature for such directed graphs,
examine its basic properties,
and conclude several geometric and analytic properties under a lower Ricci curvature bound.
Our formulation is as follows (more precisely, see Section \ref{sec:Preliminaries} and Subsection \ref{sec:Definition of Ricci curvature}):
Let $(V,\mu)$ denote a simple,
strongly connected,
finite weighted directed graph,
where $V$ is the vertex set,
and $\mu:V\times V\to [0,\infty)$ is the (non-symmetric) edge weight.
For the transition probability kernel $P:V\times V\to [0,1]$,
we consider the \textit{mean transition probability kernel} $\mathcal{P}:V\times V\to [0,1]$ defined as
\begin{equation*}\label{eq:perron vector}
\mathcal{P}(x,y):=\frac{1}{2}\left(P(x,y)+\frac{\mathfrak{m}(y)}{\mathfrak{m}(x)}P(y,x)\right),
\end{equation*}
where $\mathfrak{m}:V\to (0,1]$ is the so-called \textit{Perron measure} on $V$.
We denote by $d:V\times V\to [0,\infty)$ the (non-symmetric) distance function on $V$,
and by $W$ the associated Wasserstein distance.
For $x,y\in V$ with $x\neq y$,
we define the \textit{Ricci curvature $\kappa(x,y)$} by
\begin{equation*}
\kappa(x,y):=\lim_{\epsilon\to 0} \frac{1}{\epsilon}\left(  1-\frac{W(\nu^{\epsilon}_{x},\nu^{\epsilon}_{y})}{d(x,y)} \right),
\end{equation*}
where $\nu^{\epsilon}_{x}:V\to [0,1]$ is a probability measure on $V$ defined as
\begin{equation*}
\nu^{\epsilon}_{x}(z):=\begin{cases}
                        1-\epsilon   & \text{if $z= x$},\\
                        \epsilon\,\mathcal{P}(x,z)  & \text{if $z \neq x$}.
\end{cases}
\end{equation*}
In the undirected case (i.e., the weight $\mu$ is symmetric),
our Ricci curvature coincides with that of Lin-Lu-Yau \cite{LLY}.
We also note that the third author \cite{Y1} and Eidi-Jost \cite{EJ} have used different probability measures from $\nu^{\epsilon}_{x},\nu^{\epsilon}_{y}$ to define their Ricci curvatures (see Remark \ref{rem:various directed Ricci curvature}).

One of remarkable features of our Ricci curvature is that
it controls the behavior of the symmetric Laplacian introduced by Chung \cite{C1}, \cite{C2}.
Here we recall that
the Chung Laplacian $\mathcal{L}$ is defined as
\begin{equation*}
\mathcal{L}f(x):=f(x)-\sum_{y\in V} \mathcal{P}(x,y)f(y)
\end{equation*}
for a function $f:V\to \mathbb{R}$.
For instance,
we will derive lower bounds of the spectrum of the Chung Laplacian $\mathcal{L}$ under a lower Ricci curvature bound (see Theorems \ref{thm:main theorem} and \ref{thm:LLY comparison}).

\subsection{Main results and organization}\label{sec:Main results and organization}
In Section \ref{sec:Preliminaries},
we prepare some notations,
recall basic facts on directed graphs (see Subsections \ref{sec:Directed graphs} and \ref{sec:Laplacians}),
and optimal transport theory (see Subsection \ref{sec:Optimal transport theory}).
In Section \ref{sec:Ricci curvature},
we define our Ricci curvature (see Subsection \ref{sec:Definition of Ricci curvature}),
and examine the relation with the Chung Laplacian $\mathcal{L}$ (see Subsection \ref{sec:Ricci curvature and Laplacian}).
In Section \ref{sec:Examples},
we calculate our Ricci curvature for some concrete examples.
In Section \ref{sec:Products of directed graphs},
we further calculate it for the weighted Cartesian product of directed graphs.
In Section \ref{sec:Estimates of Ricci curvature},
we provide its upper and lower bounds.
In Section \ref{sec:Curvature-dimension conditions},
we will discuss the relation with the curvature-dimension inequality of Bakry-\'Emery type determined by the Chung Laplacian $\mathcal{L}$.

In Section \ref{sec:Comparison geometric results},
we prove several comparison geometric results under a lower Ricci curvature bound.
First,
we will extend the eigenvalue comparison of Lichnerowicz type,
and the diameter comparison of Bonnet-Myers type that have been obtained by Lin-Lu-Yau \cite{LLY} in the undirected case to our directed case (see Subsections \ref{sec:Eigenvalue comparisons} and \ref{sec:Diameter comparisons}).
Next,
we will generalize the volume comparison of Bishop type that has been established by Paeng \cite{P} in the undirected case (see Subsection \ref{sec:Volume comparisons}).
Further,
we will extend the Laplacian comparison for the distance function from a single vertex that has been established by M\"unch-Wojciechowski \cite{MW} in the undirected case (see Subsection \ref{sec:Laplacian comparisons}).

To formulate our comparison geometric results,
we introduce a notion of the asymptotic mean curvature around each vertex as follows:
For $x\in V$,
we define the \textit{asymptotic mean curvature $\mathcal{H}_{x}$ around $x$} by
\begin{equation*}
\mathcal{H}_{x}:=\mathcal{L} \rho_{x}(x),
\end{equation*}
where $\rho_{x}:V\to \mathbb{R}$ is the distance function from $x$ defined as $\rho_{x}(y):=d(x,y)$ (see Remark \ref{rem:formulation of asymptotic mean curvature} for the reason why we call it the asymptotic mean curvature).
It holds that $\mathcal{H}_{x}\leq -1$ in general.
In the undirected case,
we always have $\mathcal{H}_{x}=-1$;
in particular,
this notion plays an essential role in the case where $(V,\mu)$ is not undirected.

On Riemannian manifolds with a lower Ricci curvature bound,
it is well-known that
several comparison geometric results hold for hypersurfaces with a mean curvature bound (see the pioneering work of Heintze-Karcher \cite{HK}, and see e.g., \cite{Ba}, \cite{M}, \cite{Mo}).
In a spirit of Heintze-Karcher comparison,
for instance,
we formulate our Laplacian comparison as follows:
\begin{thm}\label{thm:Laplacian comparison}
Let $x\in V$.
For $K\in \mathbb{R}$
we assume $\inf_{y\in V\setminus \{x\}}\kappa(x,y) \geq K$.
For $\Lambda \in (-\infty,-1]$
we further assume $\mathcal{H}_{x}\geq \Lambda$.
Then on $V\setminus \{x\}$,
we have
\begin{equation}\label{eq:Laplacian comparison}
\mathcal{L} \rho_{x} \geq K \rho_{x}+\Lambda.
\end{equation}
\end{thm}

M\"unch-Wojciechowski \cite{MW} have established Theorem \ref{thm:Laplacian comparison} in the undirected case (see Theorem 4.1 in \cite{MW}).
Here we emphasize that
in the undirected case,
the lower asymptotic mean curvature bound has not been supposed since $\mathcal{H}_{x}=-1$ in that case.
We compare Theorem \ref{thm:Laplacian comparison} with a similar comparison result on Riemannian manifolds (see Subsection \ref{sec:Laplacian comparisons}).

In Section \ref{sec:Dirichlet eigenvalues of p-Laplacian},
we study the Dirichlet eigenvalues of the $p$-Laplacian $\mathcal{L}_{p}$ defined by
\begin{equation*}
\mathcal{L}_{p}f(x):=\sum_{y\in V} \vert f(x)-f(y) \vert^{p-2} (f(x)-f(y))\mathcal{P}(x,y)
\end{equation*}
for $p\in (1,\infty)$,
where we notice that $\mathcal{L}_{2}=\mathcal{L}$.
For a non-empty subset $\mathcal{V}$ of $V$ with $\mathcal{V}\neq V$,
let $\lambda^{D}_{p}(\mathcal{V})$ stand for the smallest Dirichlet eigenvalue over $\mathcal{V}$ (more precisely, see Subsection \ref{sec:Dirichlet p-Poincare constants}).
We first prove an inequality of Cheeger type for $\lambda^{D}_{p}(\mathcal{V})$ (see Subsection \ref{sec:Cheeger inequalities}).
For $x\in V$,
the \textit{inscribed radius $\IR_{x} V$ of $V$ at $x$} is defined by
\begin{equation*}\label{eq:inscribed radius}
\IR_{x} V:=\sup_{y\in V}\rho_{x}(y).
\end{equation*}
Combining the Cheeger inequality and Theorem \ref{thm:Laplacian comparison},
we obtain the following lower bound of the Dirichlet eigenvalue over the outside of a metric ball under our lower curvature bounds,
and an upper inscribed radius bound:
\begin{thm}\label{thm:main theorem}
Let $x\in V$ and $p\in (1,\infty)$.
For $K\in \mathbb{R}$
we assume $\inf_{y\in V\setminus \{x\}}\kappa(x,y) \geq K$.
For $\Lambda \in (-\infty,-1]$
we also assume $\mathcal{H}_{x}\geq \Lambda$.
For $D>0$
we further assume $\IR_{x} V\leq D$.
Then for every $R \geq 1$ with $KR+\Lambda>0$,
we have
\begin{equation}\label{eq:main theorem}
\lambda^{D}_{p}(E_{R}(x))\geq \frac{2^{p-1}}{p^{p}} \left(\frac{K R+\Lambda}{D}\right)^{p},
\end{equation}
where $E_{R}(x):=\{y\in V \mid \rho_{x}(y)\geq R\}$.
\end{thm}

We stress that
Theorem \ref{thm:main theorem} is new even in the undirected case (cf. Remark \ref{rem:isop Riem graph}).

\subsection*{{\rm Acknowledgements}}
The authors are grateful to the anonymous referees for valuable comments.
The first author was supported in part by JSPS KAKENHI (19K14532).
The first and second authors were supported in part by JSPS Grant-in-Aid for Scientific Research on Innovative Areas ``Discrete Geometric Analysis for Materials Design" (17H06460).

\section{Preliminaries}\label{sec:Preliminaries}
In this section,
we review basics of directed graphs.
We refer to \cite{G} for the notation and basics of the theory of undirected graph. 

\subsection{Directed graphs}\label{sec:Directed graphs}
Let $(G,\mu)$ be a finite weighted directed graph,
namely,
$G=(V,E)$ is a finite directed graph,
and $\mu:V \times V\to [0,\infty)$ is a function such that $\mu(x,y)>0$ if and only if $x\rightarrow y$,
where $x\rightarrow y$ means that $(x,y) \in E$.
We will denote by $n$ the cardinality of $V$.
The function $\mu$ is called the \textit{edge weight},
and we write $\mu(x,y)$ by $\mu_{xy}$.
We notice that
$(G,\mu)$ is undirected if and only if $\mu_{xy}=\mu_{yx}$ for all $x,y\in V$,
and simple if and only if $\mu_{xx}=0$ for all $x\in V$.
For $x\in V$ and $\Omega \subset V$
we set
\begin{equation*}
\mu(x):=\sum_{y\in V}\mu_{xy},\quad \mu(\Omega):=\sum_{x\in \Omega}\mu(x).
\end{equation*}
We also note that
$(G,\mu)$ has no isolated points if and only if $\mu(x)>0$ for all $x\in V$.
The weighted directed graph can be denoted by $(V,\mu)$ since $\mu$ contains full information of $E$.
Thus in this paper,
we use $(V,\mu)$ instead of $(G,\mu)$. 

For $x \in V$, 
its \textit{outer neighborhood $N_{x}$}, \textit{inner one $\olarrow{N}_{x}$}, and \textit{neighborhood $\mathcal{N}_{x}$} are defined as
\begin{equation}\label{eq:neighborhoods}
N_{x}:= \left\{y \in V \mid x\rightarrow y \right\},\quad \olarrow{N}_{x}:= \left\{y \in V \mid y \rightarrow x \right\},\quad \mathcal{N}_{x}:=N_{x} \cup \olarrow{N}_{x},
\end{equation}
respectively.
Its \textit{outer degree $\orarrow{\deg}(x)$} and \textit{inner degree $\olarrow{\deg}(x)$} are defined as the cardinality of $N_{x}$ and $\olarrow{N}_{x}$,
respectively.
We say that
$(V,\mu)$ is \textit{unweighted} if $\mu_{xy}=1$ whenever $x\rightarrow y$,
and then $\mu(x)=\orarrow{\deg}(x)$ for all $x\in V$.
In the unweighted case,
$(V,\mu)$ is said to be \textit{Eulerian} if $\orarrow{\deg}(x)=\olarrow{\deg}(x)$ for all $x\in V$.
An Eulerian graph is called \textit{regular} if $\orarrow{\deg}(x)$ (or equivalently, $\olarrow{\deg}(x)$) does not depend on $x$.
Furthermore,
for $r\geq 1$,
a regular graph is called \textit{$r$-regular} if we possess $\orarrow{\deg}(x)=r$ (or equivalently, $\olarrow{\deg}(x)=r$) for all $x\in V$.

For $x,y\in V$,
a sequence $\left\{x_{i} \right\}_{i=0}^{l}$ of vertexes is called a \textit{directed path} from $x$ to $y$ if $x_{i}\rarrow x_{i+1}$ for all $i=0,\dots,l-1$.
The number $l$ is called its length.
Furthermore,
$(V,\mu)$ is called \textit{strongly connected} if
for all $x,y\in V$,
there exists a directed path from $x$ to $y$.
Notice that
if $(V,\mu)$ is strongly connected,
then it has no isolated points.
For strongly connected $(V,\mu)$,
the (non-symmetric) distance function $d:V\times V\to [0,\infty)$ is defined as follows:
$d(x,y)$ is defined to be the minimum of the length of directed paths from $x$ to $y$.
For a fixed $x\in V$,
the \textit{distance function $\rho_{x}:V\to \mathbb{R}$},
and the \textit{reverse distance function $\olarrow{\rho}_{x}:V\to \mathbb{R}$ from $x$} are defined as
\begin{equation}\label{eq:distance function from a single point}
\rho_{x}(y):=d(x,y),\quad \olarrow{\rho}_{x}(y):=d(y,x).
\end{equation}
We further define the \textit{inscribed radius $\IR_{x} V$ of $V$ at $x$} by
\begin{equation}\label{eq:inscribed radius}
\IR_{x} V:=\sup_{y\in V}\rho_{x}(y).
\end{equation}
For $L>0$,
a function $f:V\to \mathbb{R}$ is said to be \textit{$L$-Lipschitz} if
\begin{equation*}
f(y)- f(x) \leq L\, d(x,y)
\end{equation*}
for all $x,y\in V$.
We remark that
$\rho_{x}$ is $1$-Lipschitz,
but $\olarrow{\rho}_{x}$ is not always $1$-Lipschitz.
Let $\Lip_{L}(V)$ stand for the set of all $L$-Lipschitz functions on $V$.

\begin{rem}
The non-symmetric distance function also appears in the Finsler geometry (see e.g., \cite{BCS}, \cite{Sh}).
We refer to \cite{Oh1}, \cite{Oh2} for the notation and terminology concerning the distance.
\end{rem}

\subsection{Laplacian}\label{sec:Laplacians}
Let $(V,\mu)$ be a strongly connected, finite weighted directed graph.
We recall the formulation of the Laplacian on $(V,\mu)$ introduced by Chung \cite{C1}, \cite{C2}.
The \textit{transition probability kernel} $P:V\times V\to [0,1]$ is defined as
\begin{equation}\label{eq:Markov kernel}
P(x,y):=\frac{\mu_{xy}}{\mu(x)},
\end{equation}
which is well-defined since $(V,\mu)$ has no isolated points.
Since $(V,\mu)$ is finite and strongly connected,
the Perron-Frobenius theorem implies that
there exists a unique (up to scaling) positive function $m:V\to (0,\infty)$ such that
\begin{equation}\label{eq:Perron Frobenius}
m(x)=\sum_{y\in V}m(y)P(y,x).
\end{equation}
A probability measure $\mathfrak{m}:V\to (0,1]$ on $V$ satisfying (\ref{eq:Perron Frobenius}) is called the \textit{Perron measure}.
For a non-empty subset $\Omega \subset V$,
its \textit{measure} is defined as
\begin{equation}\label{eq:measure}
\mathfrak{m}(\Omega):=\sum_{x\in \Omega}\mathfrak{m}(x).
\end{equation}

\begin{rem}\label{rem:concrete perron measure}
When $(V,\mu)$ is undirected or Eulerian,
the Perron measure $\mathfrak{m}$ is given by
\begin{equation}\label{eq:concrete perron measure}
\mathfrak{m}(x)=\frac{\mu(x)}{\mu(V)};
\end{equation}
in particular,
if $(V,\mu)$ is a regular graph,
then $\mathfrak{m}(x)=1/n$ for all $x\in V$ (see Examples 1, 2, 3 in \cite{C1}).
Here we recall that
$n$ is the cardinality of $V$.
\end{rem}

We denote by $\mathfrak{m}$ the Perron measure.
We define the \textit{reverse transition probability kernel} $\olarrow{P}:V\times V\to [0,1]$,
and the \textit{mean transition probability kernel} $\mathcal{P}:V\times V\to [0,1]$ by
\begin{equation}\label{eq:perron vector}
\olarrow{P}(x,y):=\frac{\mathfrak{m}(y)}{\mathfrak{m}(x)}P(y,x),\quad \mathcal{P}:=\frac{1}{2}(P+\olarrow{P}).
\end{equation}

\begin{rem}\label{rem:Euler mean trans}
For later convenience,
we notice the formula of the mean transition probability kernel in the case where $(V,\mu)$ is Eulerian.
When $(V,\mu)$ is Eulerian,
we deduce
\begin{equation*}
P(x,y)= \begin{cases}
		\cfrac{1}{\orarrow{\deg}(x)} & \text{if $x\rightarrow y$},\\
		0 & \text{otherwise},
		\end{cases}\quad
\olarrow{P}(x,y) = \begin{cases}
		\cfrac{1}{\orarrow{\deg}(x)} & \text{if $y\rightarrow x$},\\
		0 & \text{otherwise}
		\end{cases}
\end{equation*}
from (\ref{eq:concrete perron measure}).
Therefore
we possess
\begin{equation}\label{eq:Euler mean trans}
\mathcal{P}(x,y)= \begin{cases}
		\cfrac{1}{\orarrow{\deg}(x)} & \text{if $x\rightarrow y$ and $y\rightarrow x$},\\
		\cfrac{1}{2\,\orarrow{\deg}(x)} & \text{if either $x\rightarrow y$ or $y\rightarrow x$},\\
		0 & \text{otherwise}.
		\end{cases}
\end{equation}
\end{rem}

Let $\mathcal{F}$ stand for the set of all functions on $V$.
Chung \cite{C1}, \cite{C2} has introduced the following \textit{(positive, normalized) Laplacian} $\mathcal{L}:\mathcal{F}\to \mathcal{F}$ on $(V,\mu)$:
\begin{equation}\label{eq:Chung Laplacian}
\mathcal{L}f(x):=f(x)-\sum_{y\in V} \mathcal{P}(x,y)f(y).
\end{equation}
We will also use the \textit{negative Laplacian} $\Delta:\mathcal{F}\to \mathcal{F}$ defined by
\begin{equation}\label{eq:negative Laplacian}
\Delta:=-\mathcal{L}.
\end{equation}

The \textit{inner product} and the \textit{norm} on $\mathcal{F}$ are defined by
\begin{equation*}
(f_{0},f_{1}):=\sum_{x\in V} f_{0}(x)f_{1}(x)\mathfrak{m}(x),\quad \Vert f \Vert:=(f,f)^{1/2},
\end{equation*}
respectively.
We define a function $\mathfrak{m}:V\times V \to [0,\infty)$ by
\begin{equation}\label{eq:symmetric edge weight}
\mathfrak{m}(x,y):=\frac{1}{2}(\mathfrak{m}(x)P(x,y)+\mathfrak{m}(y)P(y,x))=\mathfrak{m}(x)\mathcal{P}(x,y).
\end{equation}
We write $\mathfrak{m}(x,y)$ by $\mathfrak{m}_{xy}$.
The following basic properties hold:
(1) $\mathfrak{m}_{xy}=\mathfrak{m}_{yx}$;
(2) $\mathfrak{m}_{xy}>0$ if and only if $y\in \mathcal{N}_{x}$ (or equivalently, $x\in \mathcal{N}_{y}$);
(3) $\mathcal{P}(x,y)=\mathfrak{m}_{xy}/\mathfrak{m}(x)$.

We also have the following integration by parts formula,
which can be proved by the same calculation as in the proof of Theorem 2.1 in \cite{G}:
\begin{prop}\label{prop:integration by parts}
Let $\Omega \subset V$ be a non-empty subset.
Then for all $f_{0},f_{1}:V\to \mathbb{R}$,
\begin{align*}
\sum_{x\in \Omega}\mathcal{L}f_{0}(x)f_{1}(x)\mathfrak{m}(x)
&=\frac{1}{2}\sum_{x,y\in \Omega}(f_{0}(y)-f_{0}(x))(f_{1}(y)-f_{1}(x))\mathfrak{m}_{xy}\\
&\quad -\sum_{x\in \Omega}\sum_{y\in V\setminus \Omega}(f_{0}(y)-f_{0}(x))f_{1}(x)\mathfrak{m}_{xy}.
\end{align*}
In particular,
\begin{equation*}
(\mathcal{L}f_{0},f_{1})=\frac{1}{2} \sum_{x,y\in V}(f_{0}(y)-f_{0}(x))(f_{1}(y)-f_{1}(x))\mathfrak{m}_{xy}=(f_{0},\mathcal{L}f_{1}).
\end{equation*}
\end{prop}

In virtue of Proposition \ref{prop:integration by parts},
$\mathcal{L}$ is symmetric with respect to the inner product.

\begin{rem}\label{rem:various directed Laplacian}
Besides the Chung Laplacian $\mathcal{L}$,
there are several generalizations of the undirected graph Laplacian for directed graphs.
For instance,
Bauer \cite{B} has studied spectral properties of the (non-symmetric) Laplace operator $\mathcal{L}_{0}:\mathcal{F}\to \mathcal{F}$ defined as
\begin{equation*}
\mathcal{L}_{0}f(x):=f(x)-\frac{1}{\sum_{y\in V}\mu_{yx}} \sum_{y\in V} \mu_{yx}f(y),
\end{equation*}
which is equivalent to the operator $\mathcal{L}_{1}:\mathcal{F}\to \mathcal{F}$ defined as
\begin{equation*}
\mathcal{L}_{1}f(x):=f(x)-\sum_{y\in V}P(x,y)f(y)
\end{equation*}
in the sense that
the spectrum of $\mathcal{L}_{0}$ on $(V,\mu)$ coincides with that of $\mathcal{L}_{1}$ on the directed graph that is obtained from $(V,\mu)$ by reversing all edges (see Definition 2.1 in \cite{B}).
On the other hand,
Yoshida \cite{Y} has recently introduced the (non-linear) \textit{submodular Laplace operator} in the context of discrete convex analysis,
which can be applied to the study of directed graphs (see Example 1.5 in \cite{Y}).
He formulated an inequality of Cheeger type for the eigenvalues of the submodular Laplace operator.
We stress that
$(V,\mu)$ does not need to be strongly connected when we define the Laplace operators in \cite{B}, \cite{Y},
unlike the Chung Laplacian.
\end{rem}

\subsection{Optimal transport theory}\label{sec:Optimal transport theory}
We recall the basic facts on the optimal transport theory,
and refer to \cite{V1}, \cite{V2}.
Let $(V,\mu)$ denote a strongly connected, finite weighted directed graph.
For two probability measures $\nu_{0},\nu_{1}$ on $V$,
a probability measure $\pi :V\times V\to [0,\infty)$ is called a \textit{coupling of $(\nu_{0},\nu_{1})$} if
\begin{equation*}
\sum_{y \in V}\pi(x,y) =\nu_{0}(x),\quad \sum_{x \in V}\pi(x,y) = \nu_{1}(y).
\end{equation*}
Let $\Pi(\nu_{0},\nu_{1})$ denote the set of all couplings of $(\nu_{0},\nu_{1})$.
The \textit{($L^{1}$-)Wasserstein distance} from $\nu_{0}$ to $\nu_{1}$ is defined as
\begin{equation}\label{eq:Wasserstein distance}
W(\nu_{0},\nu_{1}):=\inf_{\pi \in \Pi(\nu_{0},\nu_{1})} \sum_{x,y \in V}d(x,y)\pi(x,y).
\end{equation}
This is known to be a (non-symmetric) distance function on the set of all probability measures on $V$.
We also note that
$W(\delta_{x},\delta_{y})=d(x,y)$ for all $x,y\in V$,
where $\delta_{x}:V\to \{0,1\}$ denotes the Dirac measure at $x$ defined as
\begin{equation*}
\delta_{x}(z):=\begin{cases}
                        1   & \text{if $z=x$},\\
                        0   & \text{otherwise}.
                      \end{cases}
\end{equation*}
A coupling $\pi$ is called \textit{optimal} if it attains the infimum of (\ref{eq:Wasserstein distance}).
It is well-known that for any $\nu_{0},\nu_{1}$,
there exists an optimal coupling (cf. Theorem 4.1 in \cite{V2}).

The distance $W$ enjoys the following jointly convexity property (cf. Section 7.4 in \cite{V1}):
\begin{prop}\label{prop:jointly convex}
Let $t\in [0,1]$.
For any four probability measures $\nu_{0},\nu_{1},\sigma_{0},\sigma_{1}$ on $V$,
\begin{equation*}
W((1-t)\nu_{0}+t\nu_{1},(1-t)\sigma_{0}+t\sigma_{1}) \leq (1-t)\, W(\nu_{0},\sigma_{0})+t\,W(\nu_{1},\sigma_{1}).
\end{equation*}
\end{prop}

We also recall the following Kantorovich-Rubinstein duality formula (cf. Theorem 5.10 and Particular Cases 5.4 and 5.16 in \cite{V2}, and see also Subsection 2.2 in \cite{Oh2}):
\begin{prop}\label{prop:Kantorovich duality}
For any two probability measures $\nu_{0}, \nu_{1}$ on $V$, we have
\begin{equation*}
W(\nu_{0},\nu_{1}) = \sup_{f\in \Lip_{1}(V)} \sum_{x\in V} f(x)\left(  \nu_{1}(x)-\nu_{0}(x)  \right).
\end{equation*}
\end{prop}

\section{Ricci curvature}\label{sec:Ricci curvature}
In this section,
we propose a generalization of the Ricci curvature of Lin-Lu-Yau \cite{LLY} for directed graphs,
and investigate its basic properties.
In what follows,
we denote by $(V,\mu)$ a simple, strongly connected, finite weighted directed graph.

\subsection{Definition of Ricci curvature}\label{sec:Definition of Ricci curvature}
Let us introduce our Ricci curvature.
For $\epsilon \in [0,1]$ and $x\in V$,
we define a probability measure $\nu^{\epsilon}_{x}:V\to [0,1]$ by
\begin{equation}\label{eq:random walk}
\nu^{\epsilon}_{x}(z):=\begin{cases}
                        1-\epsilon   & \text{if $z= x$},\\
                        \epsilon\,\mathcal{P}(x,z)  & \text{if $z \neq x$},
                      \end{cases}
\end{equation}
which can also be written as
\begin{equation}\label{eq:modified random walk}
\nu^{\epsilon}_{x}(z)=(1-\epsilon)\delta_{x}(z)+\epsilon \,\mathcal{P}(x,z).
\end{equation}
Here $\mathcal{P}$ is defined as (\ref{eq:perron vector}).
Note that $\nu^{\epsilon}_{x}$ is a probability measure since $(V,\mu)$ is simple,
and it is supported on $\{x\}\cup \mathcal{N}_{x}$,
where $\mathcal{N}_{x}$ is defined as (\ref{eq:neighborhoods}).

We also notice the following useful property:
\begin{lem}\label{lem:integration property of random walk}
For every $f:V\to \mathbb{R}$
it holds that
\begin{equation}\label{eq:probability measure and Laplacian}
\sum_{z\in V} f(z)\nu^{\epsilon}_{x}(z)=(f+\epsilon \Delta f)(x),
\end{equation}
where $\Delta$ is defined as $(\ref{eq:negative Laplacian})$.
\end{lem}
\begin{proof}
From straightforward computations
we deduce
\begin{equation*}
\sum_{z\in V} f(z)\nu^{\epsilon}_{x}(z)=(1-\epsilon)f(x)+\epsilon \sum_{z\in V\setminus \{x\}}\mathcal{P}(x,z)f(z)=(f+\epsilon \Delta f)(x).
\end{equation*}
Here we used the simpleness of $(V,\mu)$ in the second equality.
This proves (\ref{eq:probability measure and Laplacian}).
\end{proof}

For $x,y\in V$ with $x\neq y$,
we set
\begin{equation}\label{eq:pre Ricci curvature}
\kappa_{\epsilon}(x,y):=1-\frac{W(\nu^{\epsilon}_{x},\nu^{\epsilon}_{y})}{d(x,y)},
\end{equation}
where $W$ is defined as (\ref{eq:Wasserstein distance}).
We will define our Ricci curvature as the limit of $\kappa_{\epsilon}(x,y)/\epsilon$ as $\epsilon \to 0$.
To do so,
we first verify the following (cf. Lemma 2.1 in \cite{LLY}, and see also \cite{BCLMP}, \cite{MW}):
\begin{lem}\label{lem:preconcavity}
Let $x,y\in V$ with $x\neq y$.
Then $\kappa_{\epsilon}(x,y)$ is concave in $\epsilon \in [0,1]$. 
In particular,
$\kappa_{\epsilon}(x,y)/\epsilon$ is non-increasing in $\epsilon \in (0,1]$.
\end{lem}
\begin{proof}
Fix $\epsilon_{0},\epsilon_{1},t\in [0,1]$ with $\epsilon_{0}\leq \epsilon_{1}$,
and set $\epsilon_{t}:=(1-t)\epsilon_{0} + t \epsilon_{1}$.
We can check that
\begin{equation*}
\nu^{\epsilon_t}_{x} = (1-t)\nu^{\epsilon_0}_{x} + t \nu^{\epsilon_1}_{x},\quad \nu^{\epsilon_t}_{y} = (1-t)\nu^{\epsilon_0}_{y} + t \nu^{\epsilon_1}_{y}. 
\end{equation*}
Proposition \ref{prop:jointly convex} tells us that
\begin{align}\label{eq:concavitiy of pre Ricci curvature}
\kappa_{\epsilon_t}(x,y) & = 1-\frac{W(\nu^{\epsilon_t}_{x},\nu^{\epsilon_t}_{y})}{d(x,y)} = 1-\frac{W((1-t)\nu^{\epsilon_0}_{x} + t \nu^{\epsilon_1}_{x},(1-t)\nu^{\epsilon_0}_{y} + t \nu^{\epsilon_1}_{y}    )}{d(x,y)}\\ \notag
& \geq 1-\frac{(1-t)W(\nu^{\epsilon_0}_{x},\nu^{\epsilon_0}_{y}) + t\, W(\nu^{\epsilon_1}_{x},\nu^{\epsilon_1}_{y})}{d(x,y)}= (1-t) \kappa_{\epsilon_0}(x,y) + t \kappa_{\epsilon_1}(x,y). 
\end{align}
Therefore,
we arrive at the concavity.

Applying (\ref{eq:concavitiy of pre Ricci curvature}) to $\epsilon_{0}=0$,
and noticing $\epsilon_{t}=t\epsilon_{1}$ and $\kappa_{0}(x,y)=0$,
we have
\begin{equation*}
\frac{\kappa_{\epsilon_t}(x,y)}{\epsilon_t} \ge \frac{t \kappa_{\epsilon_1}(x,y)}{t \epsilon_1} =  \frac{\kappa_{\epsilon_1}(x,y)}{\epsilon_1}
\end{equation*}
for $t\in (0,1]$,
and hence $\kappa_{\epsilon}(x,y)/\epsilon$ is non-increasing in $\epsilon \in (0,1]$.
We conclude the lemma.
\end{proof}
Lin-Lu-Yau \cite{LLY} have shown Lemma \ref{lem:preconcavity} in the undirected case (see Lemma 2.1 in \cite{LLY}).

In view of Lemma \ref{lem:preconcavity},
it suffices to show that
$\kappa_{\epsilon}(x,y)/\epsilon$ is bounded from above by a constant which does not depend on $\epsilon$.
In order to derive the boundedness,
we consider the \textit{asymptotic mean curvature $\mathcal{H}_{x}$ around $x$} that is already introduced in Subsection \ref{sec:Main results and organization},
and the \textit{reverse asymptotic mean curvature $\olarrow{\mathcal{H}}_{x}$} defined as
\begin{equation*}\label{eq:asymptotic mean curvature}
\mathcal{H}_{x}:=\mathcal{L} \rho_{x}(x),\quad \olarrow{\mathcal{H}}_{x}:=\mathcal{L} \olarrow{\rho}_{x}(x),
\end{equation*}
where $\mathcal{L}$ is defined as (\ref{eq:Chung Laplacian}),
and $\rho_{x}$ and $\olarrow{\rho}_{x}$ are done as (\ref{eq:distance function from a single point}).
More explicitly,
\begin{align}\label{eq:mean curvature formula}
\mathcal{H}_{x}&=-\sum_{y\in V} \mathcal{P}(x,y)d(x,y)=-\frac{1}{2}-\frac{1}{2}\sum_{y\in V} \olarrow{P}(x,y)d(x,y),\\ \label{eq:reverse mean curvature formula}
\olarrow{\mathcal{H}}_{x}&=-\sum_{y\in V} \mathcal{P}(x,y)d(y,x)=-\frac{1}{2}\sum_{y\in V} P(x,y)d(y,x)-\frac{1}{2}                               
\end{align}
for $P,\,\olarrow{P}$ defined as (\ref{eq:Markov kernel}), (\ref{eq:perron vector}).
We have $\mathcal{H}_{x}\leq -1$ and $\olarrow{\mathcal{H}}_{x}\leq -1$ since $\min\{d(x,y),d(y,x)\}=1$ for $y\in \mathcal{N}_{x}$.
Furthermore,
we see $\mathcal{H}_{x}=\olarrow{\mathcal{H}}_{x}= -1$ in the undirected case (see Remark \ref{rem:concrete perron measure}).

\begin{rem}\label{rem:formulation of asymptotic mean curvature}
The formulation of asymptotic mean curvature is based on the following observation concerning Riemannian geometry:
Let $(M,g)$ be a Riemannian manifold (without boundary).
We denote by $d_{g}$ the Riemannian distance,
and by $\mathcal{L}_{g}$ the Laplacian defined as the minus of the trace of Hessian.
For a fixed $x\in M$,
let $\rho_{g,x}$ stand for the distance function from $x$ defined as $\rho_{g,x}:=d_{g}(x,\cdot)$.
For a sufficiently small $R>0$,
we consider the metric sphere $S_{g,R}(x)$ with radius $R$ centered at $x$.
Then the (inward) mean curvature of $S_{g,R}(x)$ at $y \in S_{g,R}(x)$ is equal to $\mathcal{L}_{g} \rho_{g,x}(y)$.
We notice that
in the manifold case,
the mean curvature tends to $-\infty$ as $R\to 0$,
unlike the graph case.
\end{rem}

For $x,y\in V$,
we define the \textit{mixed asymptotic mean curvature} $\mathcal{H}(x,y)$ by
\begin{equation*}\label{eq:mixed asymptotic mean curvature}
\mathcal{H}(x,y):=-(\mathcal{H}_{x}+\olarrow{\mathcal{H}}_{y}).
\end{equation*}
We have $\mathcal{H}(x,y)\geq 2$;
moreover,
the equality holds in the undirected case.

We now present the following upper estimate of $\kappa_{\epsilon}(x,y)/\epsilon$ in terms of the mixed asymptotic mean curvature (cf. Lemma 2.2 in \cite{LLY}):
\begin{lem}\label{lem:preboundedness}
For all $\epsilon \in [0,1]$ and $x,y\in V$ with $x\neq y$,
we have
\begin{equation}\label{eq:preboundedness}
\frac{\kappa_{\epsilon}(x,y)}{\epsilon}\leq \frac{\mathcal{H}(x,y)}{d(x,y)}.
\end{equation}
\end{lem}
\begin{proof}
By the triangle inequality,
we have
\begin{align*}
W(\nu^{\epsilon}_{x},\nu^{\epsilon}_{y})&\geq W(\delta_{x},\delta_{y})-W(\delta_{x},\nu^{\epsilon}_{x})-W(\nu^{\epsilon}_{y},\delta_{y})\\
                                                          &   =   d(x,y)-W(\delta_{x},\nu^{\epsilon}_{x})-W(\nu^{\epsilon}_{y},\delta_{y}).
\end{align*}
From Lemma \ref{lem:integration property of random walk},
it follows that
\begin{align*}
W(\delta_{x},\nu^{\epsilon}_{x})&=\sum_{z\in V}d(x,z)\nu^{\epsilon}_{x}(z)=\epsilon \,\Delta \rho_{x}(x)=-\epsilon \,\mathcal{H}_{x},\\
W(\nu^{\epsilon}_{y},\delta_{y})&=\sum_{z\in V}d(z,y)\nu^{\epsilon}_{y}(z)=\epsilon \,\Delta \olarrow{\rho}_{y}(y)=-\epsilon \,\olarrow{\mathcal{H}}_{y}.                                             
\end{align*}
This yields
\begin{equation*}
W(\nu^{\epsilon}_{x},\nu^{\epsilon}_{y})\geq d(x,y)+\epsilon (  \mathcal{H}_{x}+\olarrow{\mathcal{H}}_{y} )=d(x,y)-\epsilon\,\mathcal{H}(x,y).
\end{equation*}
We obtain
\begin{equation*}
\frac{\kappa_{\epsilon}(x,y)}{\epsilon}=\frac{1}{\epsilon}\left(1-\frac{W(\nu^{\epsilon}_{x},\nu^{\epsilon}_{y})}{d(x,y)}   \right)\leq \frac{\mathcal{H}(x,y)}{d(x,y)}.
\end{equation*}
We complete the proof.
\end{proof}

\begin{rem}
Lin-Lu-Yau \cite{LLY} proved Lemma \ref{lem:preconcavity} in the undirected case (see Lemma 2.2 in \cite{LLY}).
We emphasize that
in the undirected case,
$\mathcal{H}(x,y)$ has not appeared in the right hand side of (\ref{eq:preboundedness}).
Actually,
its right hand side is equal to $2/d(x,y)$ in that case.
\end{rem}

In virtue of Lemmas \ref{lem:preconcavity}, \ref{lem:preboundedness},
we can define our Ricci curvature as follows:
\begin{defi}\label{defi:Ricci curvature}
For $x,y\in V$ with $x\neq y$,
we define the \textit{Ricci curvature} by
\begin{equation*}
\kappa(x,y):=\lim_{\epsilon\to 0}\frac{\kappa_{\epsilon}(x,y)}{\epsilon}.
\end{equation*}
\end{defi}

In undirected case,
this is nothing but the Ricci curvature introduced by Lin-Lu-Yau \cite{LLY}.

\begin{rem}\label{rem:various directed Ricci curvature}
Similarly to the Laplacian,
besides our Ricci curvature $\kappa(x,y)$,
there might be some generalizations of the undirected Ricci curvature of Lin-Lu-Yau \cite{LLY} for directed graphs (cf. Remark \ref{rem:various directed Laplacian}).
The third author \cite{Y1} firstly proposed the following generalization:
\begin{equation*}
\orarrow{\orarrow{\kappa}}(x,y):=\lim_{\epsilon\to 0} \frac{1}{\epsilon}\left(  1-\frac{W(\orarrow{\nu}^{\epsilon}_{x},\orarrow{\nu}^{\epsilon}_{y})}{d(x,y)} \right),
\end{equation*}
where
\begin{equation*}
\orarrow{\nu}^{\epsilon}_{x}(z):=\begin{cases}
                        1-\epsilon   & \text{if $z= x$},\\
                        \epsilon\,P(x,z)  & \text{if $z \neq x$}.
                      \end{cases}
\end{equation*}
This can be called the \textit{out-out type Ricci curvature} since we consider the Wasserstein distance from the \textit{outer probability measure} $\orarrow{\nu}^{\epsilon}_{x}$ to the \textit{outer} one $\orarrow{\nu}^{\epsilon}_{y}$.
On the other hand,
Eidi-Jost \cite{EJ} considered the \textit{in-out type Ricci curvature},
and used them for the study of directed hypergraphs (see Definition 3.2 in \cite{EJ}).
In our setting,
their in-out type Ricci curvature can be formulated as follows:
\begin{equation*}
\olarrow{\orarrow{\kappa}}(x,y):=\lim_{\epsilon\to 0}\frac{1}{\epsilon}\left(  1-\frac{W(\olarrow{\nu}^{\epsilon}_{x},\orarrow{\nu}^{\epsilon}_{y})}{d(x,y)} \right),
\end{equation*}
where
\begin{equation*}
\displaystyle\olarrow{\nu}^{\epsilon}_{x}(z):=\begin{cases}
                        1-\epsilon   & \text{if $z= x$},\\
                        \epsilon\,\displaystyle\frac{\mu_{zx}}{\sum_{y\in V}\mu_{yx}}  & \text{if $z \neq x$}.
                      \end{cases}
\end{equation*}
It seems that
we can also consider the following \textit{out-in type Ricci curvature $\orarrow{\olarrow{\kappa}}(x,y)$},
and the \textit{in-in type Ricci curvature $\olarrow{\olarrow{\kappa}}(x,y)$} defined as follows (cf. Section 8 in \cite{EJ}):
\begin{equation*}
\orarrow{\olarrow{\kappa}}(x,y):=\lim_{\epsilon\to 0} \frac{1}{\epsilon}\left(  1-\frac{W(\orarrow{\nu}^{\epsilon}_{x},\olarrow{\nu}^{\epsilon}_{y})}{d(x,y)} \right),\quad
\olarrow{\olarrow{\kappa}}(x,y):=\lim_{\epsilon\to 0} \frac{1}{\epsilon}\left(  1-\frac{W(\olarrow{\nu}^{\epsilon}_{x},\olarrow{\nu}^{\epsilon}_{y})}{d(x,y)} \right).
\end{equation*}
\end{rem}

Our Ricci curvature satisfies the following property (see Lemma 2.3 in \cite{LLY}):
\begin{prop}[\cite{LLY}]\label{prop:local to global}
Let $K\in \mathbb{R}$.
If $\kappa(z,w) \geq K$ for all edges $(z,w) \in E$,
then $\kappa(x,y) \geq K$ for any two distinct vertices $x,y\in V$.
\end{prop}
Lin-Lu-Yau \cite{LLY} obtained Proposition \ref{prop:local to global} in the undirected case (see Lemma 2.3 in \cite{LLY}).
We can prove Proposition \ref{prop:local to global} by the same argument as in the undirected case.
We omit it.

\subsection{Ricci curvature and Laplacian}\label{sec:Ricci curvature and Laplacian}
In this subsection,
we study the relation between our Ricci curvature and the Chung Laplacian $\mathcal{L}$.
In the undirected case,
M\"unch-Wojciechowski \cite{MW} have characterized the Ricci curvature of Lin-Lu-Yau \cite{LLY} in terms of the Laplacian (see Theorem 2.1 in \cite{MW}).
We will extend their characterization result to our directed setting.

Let $x,y\in V$ with $x\neq y$.
We define the \textit{gradient operator} by
\begin{equation}\label{eq:gradient operator}
\nabla_{xy} f:=\frac{f(y)-f(x)}{d(x,y)}
\end{equation}
for $f:V\to \mathbb{R}$.
Notice that
if $f$ is $L$-Lipschitz,
then $\nabla_{xy} f\leq L$.
We show the following:
\begin{lem}\label{lem:implication of Kantrovich duality}
Let $x,y\in V$ with $x \neq y$.
We have
\begin{equation}\label{eq:implication of Kantrovich duality}
\frac{\kappa_{\epsilon}(x,y)}{\epsilon}=\inf_{f\in \Lip_{1}(V)}  \left( \frac{1}{\epsilon}      \left(1-\nabla_{xy}f \right)+\nabla_{xy} \mathcal{L} f   \right).
\end{equation}
\end{lem}
\begin{proof}
Using Proposition \ref{prop:Kantorovich duality} and Lemma \ref{lem:integration property of random walk},
we have
\begin{align*}
W(\nu^{\epsilon}_{x},\nu^{\epsilon}_{y})&= \sup_{f\in \Lip_{1}(V)} \sum_{z\in V} f(z)\left(  \nu^{\epsilon}_{y}(z)-\nu^{\epsilon}_{x}(z)  \right)\\
                                                           &= \sup_{f\in \Lip_{1}(V)} \left( (f(y)+\epsilon \Delta f(y) )-(f(x)+\epsilon \Delta f(x) )  \right)\\
                                                           &= d(x,y)\,\sup_{f\in \Lip_{1}(V)}   \nabla_{xy}(f+\epsilon\,\Delta f).
\end{align*}
This leads us that
\begin{equation*}
\frac{\kappa_{\epsilon}(x,y)}{\epsilon}= \inf_{f\in \Lip_{1}(V)}  \left( \frac{1}{\epsilon}      \left(1-\nabla_{xy}(f+\epsilon\,\Delta f) \right)   \right)= \inf_{f\in \Lip_{1}(V)}  \left( \frac{1}{\epsilon}      \left(1-\nabla_{xy}f \right)+\nabla_{xy} \mathcal{L} f   \right).
\end{equation*}
We complete the proof.
\end{proof}

We set
\begin{equation*}
\mathcal{F}_{xy}:=\{f\in \Lip_{1}(V)  \mid \nabla_{xy}f=1\}.
\end{equation*}
Note that $\rho_{x}$ belongs to $\mathcal{F}_{xy}$.
We now prove the following characterization result
that has been obtained by M\"unch-Wojciechowski \cite{MW} in the undirected case (see Theorem 2.1 in \cite{MW}):
\begin{thm}\label{thm:limit free formula}
Let $x,y\in V$ with $x \neq y$.
Then we have
\begin{equation*}
\kappa(x,y) = \inf_{f\in \mathcal{F}_{xy}}  \nabla_{xy} \mathcal{L} f.
\end{equation*}
\end{thm}
\begin{proof}
We will prove along the line of the proof of Theorem 2.1 in \cite{MW}.
From Lemma \ref{lem:implication of Kantrovich duality},
\begin{equation*}
\frac{\kappa_{\epsilon}(x,y)}{\epsilon}= \inf_{f\in \Lip_{1}(V)}  \left( \frac{1}{\epsilon}      \left(1-\nabla_{xy}f \right)+\nabla_{xy} \mathcal{L} f    \right)\leq \inf_{f\in \mathcal{F}_{xy}} \nabla_{xy}\mathcal{L}f.
\end{equation*}
Letting $\epsilon \to 0$,
we obtain $\kappa(x,y) \leq  \inf_{f\in \mathcal{F}_{xy}}  \nabla_{xy} \mathcal{L} f$.

Let us show the opposite inequality.
To do so,
we consider a subset of $\Lip_{1}(V)$ defined by
\begin{equation*}
\Lip_{1,x}(V):=\{ f\in \Lip_{1}(V) \mid f(x)=0  \},
\end{equation*}
which is compact with respect to the standard topology on $\mathbb{R}^{n}$ by the finiteness of $(V,\mu)$.
We first notice that
\begin{equation}\label{eq:modified implication of Kantrovich duality}
\frac{\kappa_{\epsilon}(x,y)}{\epsilon}=\inf_{f\in \Lip_{1,x}(V)} \Phi_{\epsilon}(f)
\end{equation}
for each $\epsilon>0$,
where $\Phi_{\epsilon}$ is a functional on $\Lip_{1}(V)$ defined by
\begin{equation*}
\Phi_{\epsilon}(f):=\frac{1}{\epsilon}(1-\nabla_{xy}f)+\nabla_{xy} \mathcal{L} f.
\end{equation*}
Indeed,
we see $\Phi_{\epsilon}(f)=\Phi_{\epsilon}(f+c)$ for all $f\in \Lip_{1}(V)$ and $c\in \mathbb{R}$,
and hence the right hand side of (\ref{eq:implication of Kantrovich duality}) agrees with that of (\ref{eq:modified implication of Kantrovich duality}) by taking $c=-f(x)$.
Lemma \ref{lem:implication of Kantrovich duality} implies (\ref{eq:modified implication of Kantrovich duality}).

Now,
by the compactness of $\Lip_{1,x}(V)$,
and the continuity of the functional $\Phi_{\epsilon}$,
there exists a function $f_{\varepsilon} \in \Lip_{1,x}(V)$ which attains the infimum in the right hand side of (\ref{eq:modified implication of Kantrovich duality}).
Using the compactness of $\Lip_{1,x}(V)$ again,
we can find a sequence $\{ \epsilon_k \}_{k=1}^{\infty}$ of positive numbers with $\epsilon_k \to 0$ as $k\to\infty$ such that $f_{\epsilon_k}$ converges to some $f_0 \in \Lip_{1,x}(V)$. 
The limit of $\kappa_{\epsilon}(x,y)/\varepsilon$ exists due to Lemmas \ref{lem:preconcavity} and \ref{lem:preboundedness},
and hence (\ref{eq:modified implication of Kantrovich duality}) yields
\begin{equation*}
\nabla_{xy} f_0 = \lim_{k \to \infty} \nabla_{xy} f_{\varepsilon_k} = 1.
\end{equation*}
This means $f_0 \in \mathcal{F}_{xy}$. 
Thus we conclude
\begin{align*}
\kappa(x,y) = \lim_{\epsilon\to 0} \left(  \frac{1}{\epsilon}(1-\nabla_{xy}f_{\epsilon})+\nabla_{xy} \mathcal{L} f_{\varepsilon} \right) \geq \lim_{k\to \infty} \nabla_{xy} \mathcal{L} f_{\varepsilon_k}  = \nabla_{xy} \mathcal{L} f_0 \geq \inf_{f \in \mathcal{F}_{xy}} \nabla_{xy} \mathcal{L} f,
\end{align*}
where the first inequality follows from $\nabla_{xy} f_{\varepsilon} \leq 1$.
This completes the proof.
\end{proof}

\section{Examples}\label{sec:Examples}
In the present section,
we consider some examples,
and calculate their Ricci curvature.
In view of Proposition \ref{prop:local to global},
we only calculate the Ricci curvature for edges.
For $K\in \mathbb{R}$,
we say that
$(V,\mu)$ has \textit{constant Ricci curvature $K$} if $\kappa(x,y)=K$ for all edges $(x, y) \in E$.
In this case
we write $\kappa(V,\mu)=K$.

We first present a directed graph of positive Ricci curvature.
\begin{ex}
For $n\geq 3$,
we consider the unweighted directed complete graph with $n$ vertices,
denoted by $K_n$ (see Figure \ref{fig:complete}).
Namely,
the vertex set is $\{x_{1},\dots,x_{n}\}$,
and its edge weight is given by
\begin{align*}
\mu_{xy} := \begin{cases}
		0 & \text{if $x=x_{i+1}$ and $y=x_{i}$ for $i=1,\dots,n-1$},\\
		0 & \text{if $x=x_{1}$ and $y=x_{n}$}, \\
		1 & \text{otherwise}.
		\end{cases}
\end{align*}
Then we have the following:
(1) $\kappa(K_{3})=3/2$;
(2) if $n=4$,
then we have
\begin{align*}
\kappa (x_1, x_i) = \begin{cases}
		1 & \text{if $i = 2$},\\
		\cfrac{3}{2} & \text{if $i=3$}; 
		\end{cases}
\end{align*}
(3) if $n=5$,
then we have
\begin{align*}
\kappa (x_1, x_i) = \begin{cases}
		1 & \text{if $i = 2$},\\
		\cfrac{7}{6} & \text{if $i=3$ or $4$}; 
		\end{cases}
\end{align*}
(4) if $n \geq 6$,
then we have
\begin{align*}
\kappa (x_1, x_i) = \begin{cases}
		1 & \text{if $i = 2$ or $i\in \{4, \dots, n-2\}$},\\
		1+\cfrac{1}{2 (n - 2)} & \text{if $i=3$ or $n-1$}.
		\end{cases}
\end{align*}
We explain the method of the proof of $\kappa(x_{1},x_{2})=1$ for $n\geq 4$.
Since this graph is $(n-2)$-regular,
the formula (\ref{eq:Euler mean trans}) yields
\begin{align*}
\nu^\epsilon_{x_1}(x_i) = \begin{cases}
		1 - \epsilon & \text{if $i = 1$,}\\
		\cfrac{\epsilon}{2 (n - 2)} & \text{if $i=2$ or $n$,}\\
		\cfrac{\epsilon}{n - 2} & \text{otherwise,}
		\end{cases}\quad
\nu^\epsilon_{x_2}(x_i) = \begin{cases}
		1 - \epsilon & \text{if $i = 2$,}\\
		\cfrac{\epsilon}{2 (n - 2)} & \text{if $i=1$ or $3$,}\\
		\cfrac{\epsilon}{n - 2} & \text{otherwise.}
		\end{cases}
\end{align*}
We define a coupling $\pi$ of $(\nu^{\epsilon}_{x_{1}},\nu^{\epsilon}_{x_{2}})$ as
\begin{align*}
\pi(z,w) := \begin{cases}
		(1- \epsilon) -\cfrac{\epsilon}{2(n-2)} & \text{if $(z,w)=(x_{1},x_{2})$},\\
		\cfrac{\epsilon}{2(n-2)} & \text{if $(z,w)=(x_{3},x_{n})$ or $(x_{1},x_{1})$ or $(x_{2},x_{2})$},\\
		\cfrac{\epsilon}{n-2} & \text{if $(z,w)=(x_{i},x_{i})$ for $i=3,\dots,n-1$},\\
		0 & \text{otherwise}. 
		\end{cases}
\end{align*}
Using this coupling and (\ref{eq:Wasserstein distance}),
we see $W(\nu^{\epsilon}_{x_{1}},\nu^{\epsilon}_{x_{2}})\leq 1-\epsilon$,
and hence $\kappa(x_{1},x_{2})\geq 1$.
Further,
we define a $1$-Lipschitz function $f:V \to \mathbb{R}$ by
\begin{align*}
f(z):=\begin{cases}
		1 & \text{if $z=x_{2}$ or $x_{n}$},\\
		0 & \text{otherwise}.
		\end{cases}
\end{align*}
By applying Proposition \ref{prop:Kantorovich duality} to $f$,
we obtain $\kappa(x_{1},x_{2})\leq 1$.
It follows that
$\kappa(x_{1},x_{2})=1$.
The other parts can be proved by the same argument,
and the proof is left to the readers.
By using (\ref{eq:Euler mean trans}),
we can also calculate that
for all $n\geq 3$,
and for all $i=1,\dots,n$
\begin{equation*}
\mathcal{H}_{x_{i}}=\olarrow{\mathcal{H}}_{x_{i}}=-\left(  1+\cfrac{1}{2 (n - 2)}  \right).
\end{equation*}
\end{ex}

\begin{figure}[http]
        \begin{center} 
          \includegraphics[scale=0.7]{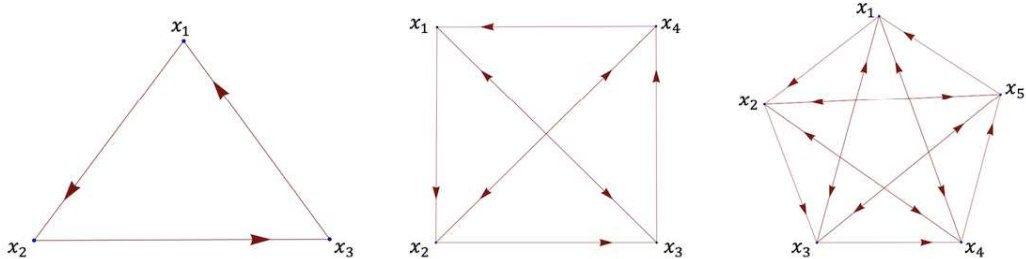}
          \caption{Directed complete graphs}
          \label{fig:complete}
        \end{center}
\end{figure}

We next present a flat directed graph.
\begin{ex}
For $n \geq 4$,
we consider the unweighted directed cycle with $n$ vertices,
denoted by $C_n$ (see Figure \ref{fig:cycle}).
Then $\kappa (C_n) = 0$.
For $i=1,\dots,n$,
let $x_{i}$ be vertexes as in Figure \ref{fig:cycle}.
It suffices to show that $\kappa(x_{1},x_{2})=0$.
Since this graph is $1$-regular,
(\ref{eq:Euler mean trans}) implies
\begin{align*}
\nu^\epsilon_{x_1}(x_i) = \begin{cases}
		1 - \epsilon & \text{if $i = 1$},\\
		\cfrac{\epsilon}{2} & \text{if $i=2$ or $n$},\\
		0 & \text{otherwise},
		\end{cases}\quad
\nu^\epsilon_{x_2}(x_i) = \begin{cases}
		1 - \epsilon & \text{if $i = 2$},\\
		\cfrac{\epsilon}{2} & \text{if $i=1$ or $3$},\\
		0 & \text{otherwise}.
		\end{cases}
\end{align*}
We define a coupling $\pi$ of $(\nu^{\epsilon}_{x_{1}},\nu^{\epsilon}_{x_{2}})$ as
\begin{align*}
\pi(z,w):= \begin{cases}
		1 - \epsilon & \text{if $(z,w)=(x_{1},x_{2})$},\\
		\cfrac{\epsilon}{2} & \text{if $(z,w)=(x_2,x_{3})$ or $(x_n,x_1)$},\\
		0 & \text{otherwise}. 
		\end{cases}
\end{align*}
By using this coupling and $(\ref{eq:Wasserstein distance})$
we possess $W(\nu^{\epsilon}_{x_{1}},\nu^{\epsilon}_{x_{2}})\leq 1$,
and thus $\kappa(x_{1},x_{2})\geq 0$.
We also define a $1$-Lipschitz function $f:V \to \mathbb{R}$ as
\begin{align*}
f(z) := \begin{cases}
		1 & \text{if $z= x_2$},\\
		0 & \text{if $z= x_1$},\\
		-1 & \text{if $z = x_n$},\\
		2 & \text{otherwise}.
		\end{cases}
\end{align*}
Applying Proposition \ref{prop:Kantorovich duality} to $f$ leads to $\kappa(x_{1},x_{2})\leq 0$.
We obtain $\kappa(x_{1},x_{2})= 0$.
From (\ref{eq:Euler mean trans}),
one can also deduce that
for all $i=1,\dots,n$,
\begin{equation*}
\mathcal{H}_{x_{i}}=\olarrow{\mathcal{H}}_{x_{i}}=-\frac{n}{2}.
\end{equation*}
\end{ex}

\begin{figure}[http]
        \begin{center} 
          \includegraphics[scale=0.7]{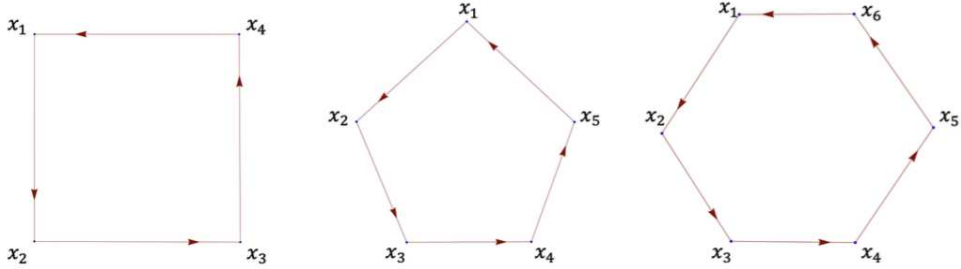}
          \caption{Directed cycles}
          \label{fig:cycle}
        \end{center}
\end{figure}

We also provide a directed graph with negatively curved edges.
\begin{ex}
We consider the unweighted directed graph as shown in Figure \ref{fig:negative}.
Then we have $\kappa(x,y)=-2$ for the vertexes $x,y$ as shown in Figure \ref{fig:negative}.
We can show this estimate as follows:
Since this graph is Eulerian,
the formula (\ref{eq:Euler mean trans}) tells us that
the two probability measures $\nu^{\epsilon}_x,\nu^{\epsilon}_y$ are given by
\begin{align*}
\nu^\epsilon_{x}(z) = \begin{cases}
		1 - \epsilon & \text{if $z= x$,}\\
		\cfrac{\epsilon}{4} & \text{if $z \in \mathcal{N}_{x}$,}\\
		0 & \text{otherwise},
		\end{cases}\quad 
\nu^\epsilon_{y}(z) = \begin{cases}
		1 - \epsilon & \text{if $z= y$,}\\
		\cfrac{\epsilon}{4} & \text{if $z \in \mathcal{N}_{y}$,}\\
		0 & \text{otherwise}.
		\end{cases}
\end{align*}
For $i=1,2,3$,
let $x_{i},y_{i}$ be as in Figure \ref{fig:negative}.
We define a coupling $\pi$ of $(\nu^{\epsilon}_{x},\nu^{\epsilon}_{y})$ by
\begin{align*}
\pi(z,w):=\begin{cases}
		1 - \epsilon & \text{if $(z,w)=(x,y)$},\\
		\cfrac{\epsilon}{4} & \text{if $(z,w)=(x_1,y_1)$ or $(x_2,y_{2})$ or $(x_3,x)$ or $(y,y_{3})$},\\
		0 & \text{otherwise}. 
		\end{cases}
\end{align*}
Then one can prove $W(\nu^{\epsilon}_{x},\nu^{\epsilon}_{y})\leq 1+2\epsilon$ by $(\ref{eq:Wasserstein distance})$,
and hence $\kappa(x,y)\geq -2$.
To check the opposite inequality,
we define a function $f:\mathcal{N}_{x} \cup \mathcal{N}_{y}\to \mathbb{R}$ by
\begin{align*}
f(z):=\begin{cases}
		7 & \text{if $z = y_3$},\\
		5 & \text{if $z=y_{1}$ or $y_{2}$},\\
		4 & \text{if $z = y$},\\
		3 & \text{if $z = x$},\\
		2 & \text{if $z=x_{1}$ or $x_{2}$},\\
		0 & \text{if $z=x_{3}$},\\
		\end{cases}
\end{align*}
which is $1$-Lipschitz on $\mathcal{N}_{x} \cup \mathcal{N}_{y}$.
We can extend $f$ to a $1$-Lipschitz function over $V$.
Applying Proposition \ref{prop:Kantorovich duality} to it,
we obtain $\kappa(x,y)\leq-2$.
This proves $\kappa(x,y)=-2$.
\end{ex}

\begin{figure}
        \begin{center} 
          \includegraphics[scale=1.0]{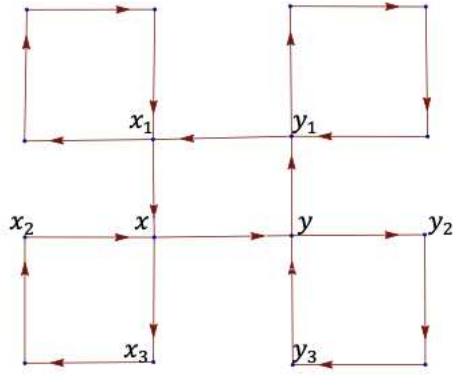}
          \caption{Directed graph with negatively curved edges}
          \label{fig:negative}
        \end{center}
\end{figure}

\section{Products}\label{sec:Products of directed graphs}
This section is devoted to the calculation of our Ricci curvature for the weighted Cartesian product of directed graphs.
We consider the weighted Cartesian product of $(V,\mu)$,
and another simple, strongly connected, finite weighted directed graph $(V',\mu')$.

\subsection{Weighted Cartesian products}
We recall the notion of the weighted Cartesian product.
For two parameters $\alpha, \beta>0$,
the \textit{$(\alpha,\beta)$-weighted Cartesian product $(V, \mu) \square_{(\alpha,\beta)} (V', \mu')$ of $(V,\mu)$ and $(V',\mu')$} is defined as follows (cf. Subsection 2.6 in \cite{C-book}, Definition 2.17 in \cite{G}, and Remark \ref{rem:weighted Cart} below):
Its vertex set is $V\times V'$,
and its edge weight is given by
\begin{equation}\label{eq:product weight}
\mu_{(\alpha,\beta)}(\mathbf{x},\mathbf{y}):=\beta\mu'(x')\mu_{xy}\,\delta_{x'}(y')+\alpha\mu(x)\mu'_{x'y'}\,\delta_x(y)
\end{equation}
for $\mathbf{x}=(x,x'), \mathbf{y}=(y,y') \in V\times V'$,
which can also be written as
\begin{equation*}
\mu_{(\alpha,\beta)}(\mathbf{x},\mathbf{y})=\begin{cases}
                \beta \mu'(x')\mu_{xy} & \text{if $x'=y'$},\\
		\alpha\mu(x)\mu'_{x'y'} & \text{if $x=y$},\\
		0 & \text{otherwise}.
		\end{cases}
\end{equation*}
Here $\mu'(x')$ denotes the vertex weight at $x'$ on $(V',\mu')$,
and we will denote by $\mu_{(\alpha,\beta)}(\mathbf{x})$ the vertex weight at $\mathbf{x}$ on $(V, \mu) \square_{(\alpha,\beta)} (V', \mu')$.
We immediately see that
$\mu_{(\alpha,\beta)}(\mathbf{x},\mathbf{x})=0$ since $(V,\mu)$ and $(V',\mu')$ are simple;
in particular,
this weighted Cartesian product is also simple.
We also observe that
$\mu_{(\alpha,\beta)}(\mathbf{x},\mathbf{y})>0$ if and only if either
(1) $\mu'_{x'y'}>0$ and $x=y$;
or (2) $\mu_{xy}>0$ and $x'=y'$,
and hence
$\mathbf{x}\rightarrow \mathbf{y}$ if and only if either
(1) $x' \rightarrow y'$ and $x=y$;
or (2) $x \rightarrow y$ and $x'=y'$.
Therefore,
the strongly connectedness of $(V,\mu)$ and $(V',\mu')$ tells us that
this weighted Cartesian product is also strongly connected.
Moreover,
its distance function $d_{(\alpha,\beta)}:(V\times V')\times (V\times V')\to [0,\infty)$ can be expressed as
\begin{equation}\label{eq:product distance}
d_{(\alpha,\beta)}(\mathbf{x},\mathbf{y})=d(x,y)+d'(x',y')
\end{equation}
for the distance functions $d$ and $d'$ on $(V,\mu)$ and $(V',\mu')$,
respectively.
In other words,
$d_{(\alpha,\beta)}$ is the $l^{1}$-distance function.

\begin{rem}\label{rem:weighted Cart}
In \cite{C-book},
the weighted Cartesian product has been formulated as $(V, \mu) \square_{(\alpha,\beta)} (V', \mu')$ with $\alpha \equiv 1$ and $\beta \equiv 1$ in our notation (see Subsection 2.6 in \cite{C-book}).
Also,
in \cite{G},
it has been done as $(V, \mu) \square_{(\alpha,\beta)} (V', \mu')$,
which generalizes that in \cite{C-book} (see Definition 2.7 in \cite{G}).
When $(V,\mu)$ is $r$-regular and $(V',\mu')$ is $r'$-regular,
the formulation in \cite{G} agrees with the standard Cartesian product by taking $\alpha=1/r$ and $\beta=1/r'$ (see Lemma 2.19 in \cite{G}).
\end{rem}

We now summarize some formulas on the weighted Cartesian product $(V, \mu) \square_{(\alpha,\beta)} (V', \mu')$ (cf. Lemma 2.18 in \cite{G}).
We denote by $P', \mathfrak{m}', \olarrow{P}', \mathcal{P}'$
the transition probability kernel,
the Perron measure,
the reverse transition probability kernel,
the mean transition probability kernel on $(V',\mu')$,
respectively.
Similarly,
we denote by $P_{(\alpha,\beta)}, \mathfrak{m}_{(\alpha,\beta)}, \olarrow{P}_{(\alpha,\beta)}, \mathcal{P}_{(\alpha,\beta)}$
the transition probability kernel,
the Perron measure,
the reverse transition probability kernel,
the mean transition probability kernel on $(V, \mu) \square_{(\alpha,\beta)} (V', \mu')$,
respectively.
\begin{lem}\label{lem:product formula}
For $\mathbf{x}=(x,x'), \mathbf{y}=(y,y') \in V\times V'$,
we have
\begin{align}\label{eq:product vertex weight}
\mu_{(\alpha,\beta)}(\mathbf{x})& = (\alpha+\beta)\mu(x)\mu'(x'),\\ \label{eq:product trans prob}
P_{(\alpha,\beta)}(\mathbf{x} ,\mathbf{y} )&= \frac{\beta}{\alpha+\beta}P(x,y)\delta_{x'}(y') + \frac{\alpha}{\alpha+\beta}P'(x',y')\delta_{x}(y),\\ \label{eq:product Perron}
\mathfrak{m}_{(\alpha,\beta)}(\mathbf{x})&= \mathfrak{m}(x) \mathfrak{m}'(x'),\\ \label{eq:product reverse prob}
\olarrow{P}_{(\alpha,\beta)}(\mathbf{x} ,\mathbf{y})&=\frac{\beta}{\alpha+\beta}\olarrow{P}(x,y)\delta_{x'}(y') +\frac{\alpha}{\alpha+\beta}\olarrow{P}'(x',y') \delta_x(y),\\ \label{eq:product mean prob}
\mathcal{P}_{(\alpha,\beta)}(\mathbf{x} ,\mathbf{y}) &=\frac{\beta}{\alpha+\beta} \mathcal{P}(x,y)\delta_{x'}(y') + \frac{\alpha}{\alpha+\beta}\mathcal{P}'(x',y')\delta_{x}(y).
\end{align}
\end{lem}
\begin{proof}
We begin with the proof of (\ref{eq:product vertex weight}).
By (\ref{eq:product weight}),
it holds that
\begin{align*}
\mu_{(\alpha,\beta)}(\mathbf{x})&=\sum_{\mathbf{z}\in V\times V'}\mu_{(\alpha,\beta)}(\mathbf{x},\mathbf{z})
                                                     =\sum_{z\in V,z'\in V'}\left(\beta\mu'(x')\mu_{xz}\,\delta_{x'}(z') +\alpha \mu(x)\mu'_{x'z'}\,\delta_x(z) \right)\\
                                                  &=\beta\mu'(x') \sum_{z\in V}\mu_{xz}+\alpha\mu(x) \sum_{z'\in V'} \mu'_{x'z'}=(\beta+\alpha)\mu(x)\mu'(x')
\end{align*}
for $\mathbf{z}=(z,z')\in V\times V'$,
and hence (\ref{eq:product vertex weight}).
Further,
(\ref{eq:product trans prob}) directly follows from (\ref{eq:product weight}), (\ref{eq:product vertex weight}).

We next show (\ref{eq:product Perron}).
We define a probability measure $\widetilde{\mathfrak{m}}:V\times V'\to (0,1]$ by
\begin{equation*}
\widetilde{\mathfrak{m}}(\mathbf{z})=\mathfrak{m}(z) \mathfrak{m}'(z').
\end{equation*}
In virtue of the uniqueness of the Perron measure,
it suffices to verify
\begin{equation}\label{eq:check product Perron}
\widetilde{\mathfrak{m}}(\mathbf{x})=\sum_{\mathbf{z}\in V\times V'}\widetilde{\mathfrak{m}}(\mathbf{z})P_{(\alpha,\beta)}(\mathbf{z},\mathbf{x}).
\end{equation}
Using (\ref{eq:product trans prob}),
we can calculate the right hand side of (\ref{eq:check product Perron}) as follows:
\begin{align}\notag
&\quad\,\, \sum_{\mathbf{z}\in V\times V'}\widetilde{\mathfrak{m}}(\mathbf{z})P_{(\alpha,\beta)}(\mathbf{z},\mathbf{x})\\ \notag
& = \sum_{z\in V,z'\in V'} \mathfrak{m}(z) \mathfrak{m}'(z')\left( \frac{\beta}{\alpha+\beta}P(z,x)\delta_{x'}(z') +\frac{\alpha}{\alpha+\beta}P'(z',x')\delta_{x}(z) \right)\\ \notag 
& =\frac{\beta}{\alpha+\beta}\mathfrak{m}'(x')\sum_{z\in V} \mathfrak{m}(z) P(z,x) + \frac{\alpha}{\alpha+\beta}\mathfrak{m}(x)\sum_{z' \in V'} \mathfrak{m}'(z') P'(z',x') \\  \label{eq:check product Perron 2}
& =\frac{\beta}{\alpha+\beta}\mathfrak{m}'(x')\mathfrak{m}(x) +\frac{\alpha}{\alpha+\beta}\mathfrak{m}(x) \mathfrak{m}'(x')= \mathfrak{m}(x) \mathfrak{m}'(x')=\widetilde{\mathfrak{m}}(\mathbf{x}),
\end{align}
and thus (\ref{eq:check product Perron}).
Here we used the fact that
$\mathfrak{m}$ and $\mathfrak{m}'$ are the Perron measures in (\ref{eq:check product Perron 2}).

Let us prove (\ref{eq:product reverse prob}).
Combining (\ref{eq:product trans prob}) and (\ref{eq:product Perron}),
we have
\begin{align*}
\olarrow{P}_{(\alpha,\beta)}(\mathbf{x} ,\mathbf{y})&=\frac{  \mathfrak{m}_{(\alpha,\beta)}(\mathbf{y})}{\mathfrak{m}_{(\alpha,\beta)}(\mathbf{x})}P_{(\alpha,\beta)}(\mathbf{y} ,\mathbf{x} )\\
&=\frac{\mathfrak{m}(y) \mathfrak{m}'(y')}{\mathfrak{m}(x) \mathfrak{m}'(x')}\left( \frac{\beta}{\alpha+\beta} P(y,x)\delta_{y'}(x') + \frac{\alpha}{\alpha+\beta}P'(y',x')\delta_{y}(x)\right)\\
&=\frac{\beta}{\alpha+\beta}\olarrow{P}(x,y)  \frac{\mathfrak{m}'(y')}{\mathfrak{m}'(x')}   \delta_{y'}(x') +\frac{\alpha}{\alpha+\beta}\olarrow{P}'(x',y')\frac{\mathfrak{m}(y)}{\mathfrak{m}(x)} \delta_{y}(x)\\
&=\frac{\beta}{\alpha+\beta}\olarrow{P}(x,y) \delta_{x'}(y') + \frac{\alpha}{\alpha+\beta}\olarrow{P}'(x',y')\delta_{x}(y).
\end{align*}
This proves (\ref{eq:product reverse prob}).
The last one (\ref{eq:product mean prob}) is an immediate consequence of (\ref{eq:product trans prob}) and (\ref{eq:product reverse prob}).
\end{proof}

Let $\epsilon \in [0,1]$ and $\mathbf{x}=(x,x')\in V\times V'$.
We next examine a probability measure $\nu_{(\alpha,\beta),\mathbf{x}}^{\varepsilon}:V\times V'\to [0,1]$ on $(V, \mu) \square_{(\alpha,\beta)} (V', \mu')$ defined as (\ref{eq:random walk}).
Let $(\nu^{\epsilon}_{x'})':V'\to [0,1]$ stand for a probability measure over $(V',\mu')$ defined as (\ref{eq:random walk}).
\begin{lem}\label{lem:product prob meas}
Let $\epsilon \in [0,1]$ and $\mathbf{x}=(x,x')\in V\times V'$.
For every $\mathbf{z}=(z,z')\in V\times V'$,
\begin{equation*}
\nu_{(\alpha,\beta),\mathbf{x}}^{\varepsilon}(\mathbf{z})=\frac{\beta}{\alpha+\beta}\nu^{\epsilon}_{x|x'}(\mathbf{z})+\frac{\alpha}{\alpha+\beta}(\nu^{\epsilon}_{x'|x})'(\mathbf{z}),
\end{equation*}
where two probability measures $\nu^{\epsilon}_{x|x'},(\nu^{\epsilon}_{x'|x})':(V\times V')\times (V\times V')\to [0,1]$ are defined as
\begin{equation}\label{eq:product sub prob}
\nu^{\epsilon}_{x|x'}(\mathbf{z}):=\nu^{\epsilon}_{x}(z)\delta_{x'}(z'),\quad (\nu^{\epsilon}_{x'|x})'(\mathbf{z}):=(\nu^{\epsilon}_{x'})'(z')\delta_{x}(z).
\end{equation}
\end{lem}
\begin{proof}
In view of the expression (\ref{eq:modified random walk}),
the formula (\ref{eq:product mean prob}) implies
\begin{align*}
\nu_{(\alpha,\beta),\mathbf{x}}^{\varepsilon}(\mathbf{z})&=(1-\epsilon)\delta_{\mathbf{x}}(\mathbf{z})+\epsilon\,\mathcal{P}_{(\alpha,\beta)}(\mathbf{x},\mathbf{z})\\
                                                                  &=(1-\epsilon)\delta_{x}(z)\delta_{x'}(z')
                                                                   +\epsilon\,\left( \frac{\beta}{\alpha+\beta} \mathcal{P}(x,z)\delta_{x'}(z') +\frac{\alpha}{\alpha+\beta}\mathcal{P}'(x',z')\delta_{x}(z) \right)\\
                                                                  &=\frac{\beta}{\alpha+\beta}((1-\epsilon)\delta_{x}(z)+\epsilon\,\mathcal{P}(x,z))\delta_{x'}(z')\\
                                                                  &\quad +\frac{\alpha}{\alpha+\beta}((1-\epsilon)\delta_{x'}(z')+\epsilon\,\mathcal{P}'(x',z'))\delta_{x}(z)\\
                                                                  &=\frac{\beta}{\alpha+\beta}\nu^{\epsilon}_{x|x'}(\mathbf{z})+\frac{\alpha}{\alpha+\beta}(\nu^{\epsilon}_{x'|x})'(\mathbf{z}).
\end{align*}
This completes the proof.
\end{proof}

We further investigate the Laplacian.
Let $\mathcal{L}'$ and $\mathcal{L}_{(\alpha,\beta)}$ be the Laplacian over $(V',\mu')$ and $(V,\mu) \square_{(\alpha,\beta)} (V',\mu')$,
respectively.
\begin{lem}\label{lem:product Laplacian}
For $f:V\to \mathbb{R}$ and $f':V'\to \mathbb{R}$,
we define a function $\mathbf{f}:V\times V'\to \mathbb{R}$ by
\begin{equation*}
\mathbf{f}(\mathbf{x}):=f(x)+f'(x')
\end{equation*}
for $\mathbf{x}=(x,x')\in V\times V'$.
Then for every $\mathbf{x}=(x,x')\in V\times V'$
we have
\begin{equation*}
\mathcal{L}_{(\alpha,\beta)}\mathbf{f}(\mathbf{x})=\frac{\beta}{\alpha+\beta}\mathcal{L}f(x)+\frac{\alpha}{\alpha+\beta}\mathcal{L}'f'(x').
\end{equation*}
\end{lem}
\begin{proof}
The formula (\ref{eq:product mean prob}) yields
\begin{align*}
&\quad\,\, \sum_{\mathbf{y}\in V\times V'}\mathcal{P}_{(\alpha,\beta)}(\mathbf{x},\mathbf{y})\mathbf{f}(\mathbf{y})\\
&=\sum_{y\in V,y'\in V'}\left(\frac{\beta}{\alpha+\beta}\mathcal{P}(x,y)\delta_{x'}(y') +\frac{\alpha}{\alpha+\beta}\mathcal{P}'(x',y')\delta_{x}(y)  \right)(f(y)+f'(y'))\\
&=\frac{\beta}{\alpha+\beta}\sum_{y\in V}\mathcal{P}(x,y)(f(y)+f'(x')) +\frac{\alpha}{\alpha+\beta}\sum_{y'\in V'} \mathcal{P}'(x',y')(f(x)+f'(y'))\\
&=\frac{\alpha}{\alpha+\beta}f(x)+\frac{\beta}{\alpha+\beta}f'(x')
 +\frac{\beta}{\alpha+\beta} \sum_{y\in V}\mathcal{P}(x,y)f(y) +\frac{\alpha}{\alpha+\beta}\sum_{y'\in V'} \mathcal{P}'(x',y')f'(y').
\end{align*}
The above equality implies
\begin{align*}
\mathcal{L}_{(\alpha,\beta)}\mathbf{f}(\mathbf{x})&=\mathbf{f}(\mathbf{x})-\sum_{\mathbf{y}\in V\times V'}\mathcal{P}_{(\alpha,\beta)}(\mathbf{x},\mathbf{y})\mathbf{f}(\mathbf{y})\\
&=\frac{\beta}{\alpha+\beta}\left(f(x)- \sum_{y\in V}\mathcal{P}(x,y)f(y)    \right) +\frac{\alpha}{\alpha+\beta}\left(f'(x')- \sum_{y'\in V}\mathcal{P}'(x',y')f'(y')    \right)\\
&=\frac{\beta}{\alpha+\beta} \mathcal{L}f(x)+\frac{\alpha}{\alpha+\beta}\mathcal{L}'f'(x').
\end{align*}
We arrive at the desired formula.
\end{proof}

We end this subsection with formulas for asymptotic mean curvature.
Let $\mathbf{x}=(x,x'), \mathbf{y}=(y,y') \in V\times V'$.
We denote by $\mathcal{H}'_{x'},\olarrow{\mathcal{H}}'_{x'},\mathcal{H}'(x',y')$
the asymptotic mean curvature around $x'$,
the reverse asymptotic mean curvature,
the mixed asymptotic mean curvature over $(V',\mu')$,
respectively.
Also,
let $\mathcal{H}_{(\alpha,\beta),\mathbf{x}},\olarrow{\mathcal{H}}_{(\alpha,\beta),\mathbf{x}},\mathcal{H}_{(\alpha,\beta)}(\mathbf{x},\mathbf{y})$ stand for
the asymptotic mean curvature around $\mathbf{x}$,
the reverse asymptotic mean curvature,
the mixed asymptotic mean curvature over $(V, \mu) \square_{(\alpha,\beta)} (V', \mu')$,
respectively.
\begin{prop}\label{prop:product mean curv}
For $\mathbf{x}=(x,x'), \mathbf{y}=(y,y') \in V\times V'$,
we have
\begin{align*}
\mathcal{H}_{(\alpha,\beta),\mathbf{x}}&=\frac{\beta}{\alpha+\beta}\mathcal{H}_{x}+\frac{\alpha}{\alpha+\beta}\mathcal{H}'_{x'},\\
\olarrow{\mathcal{H}}_{(\alpha,\beta),\mathbf{x}}&=\frac{\beta}{\alpha+\beta}\olarrow{\mathcal{H}}_{x}+\frac{\alpha}{\alpha+\beta}\olarrow{\mathcal{H}}'_{x'},\\
\mathcal{H}_{(\alpha,\beta)}(\mathbf{x},\mathbf{y})&=\frac{\beta}{\alpha+\beta}\mathcal{H}(x,y)+\frac{\alpha}{\alpha+\beta}\mathcal{H}'(x',y').
\end{align*}
\end{prop}
\begin{proof}
These formulas directly follow from (\ref{eq:product distance}) and Lemma \ref{lem:product Laplacian}.
\end{proof}

\subsection{Ricci curvature of weighted Cartesian products}
We calculate the Ricci curvature of the $(\alpha,\beta)$-weighted Cartesian product $(V, \mu) \square_{(\alpha,\beta)} (V', \mu')$ introduced in the above subsection.
For $x',y' \in V'$ with $x'\neq y'$,
let $\kappa'(x',y')$ denote the Ricci curvature over $(V',\mu')$ defined as Definition \ref{defi:Ricci curvature}.
For $\mathbf{x},\mathbf{y} \in V\times V'$ with $\mathbf{x}\neq \mathbf{y}$,
we also denote by $\kappa_{(\alpha,\beta)}(\mathbf{x},\mathbf{y})$ the Ricci curvature over $(V, \mu) \square_{(\alpha,\beta)} (V', \mu')$.
We obtain the following (cf. Theorem 3.1 in \cite{LLY}):
\begin{thm}\label{thm:product Ricci}
Let $\mathbf{x}=(x,x'), \mathbf{y}=(y,y') \in V\times V'$ with $\mathbf{x}\neq \mathbf{y}$.
Then we have the following:
\begin{enumerate}
\item If $x\neq y$ and $x'\neq y'$,
then
\begin{equation}\label{eq:product Ricci}
\kappa_{(\alpha,\beta)}(\mathbf{x},\mathbf{y}) = \frac{\beta}{\alpha+\beta} \frac{d(x,y)}{d(x,y) + d'(x',y')} \kappa(x,y)+\frac{\alpha}{\alpha+\beta}\frac{d'(x',y')}{d(x,y) + d'(x',y')} \kappa'(x',y');
\end{equation}
\item if $x\neq y$ and $x'= y'$,
then
\begin{equation}\label{eq:pre product Ricci1}
\kappa_{(\alpha,\beta)}(\mathbf{x},\mathbf{y}) = \frac{\beta}{\alpha+\beta}\kappa(x,y);
\end{equation}
\item if $x= y$ and $x'\neq y'$,
then
\begin{equation}\label{eq:pre product Ricci2}
\kappa_{(\alpha,\beta)}(\mathbf{x},\mathbf{y}) = \frac{\alpha}{\alpha+\beta} \kappa'(x',y').
\end{equation}
\end{enumerate}
\end{thm}
\begin{proof}
We only prove the formula (\ref{eq:product Ricci}).
The others (\ref{eq:pre product Ricci1}) and (\ref{eq:pre product Ricci2}) can be proved by the same argument as in the proof of (\ref{eq:product Ricci}),
and more easily.
We assume
$x\neq y$ and $x'\neq y'$.

We first show that
\begin{equation}\label{eq:left product Ricci}
\kappa_{(\alpha,\beta)}(\mathbf{x},\mathbf{y}) \geq \frac{\beta}{\alpha+\beta} \frac{d(x,y)}{d(x,y) + d'(x',y')} \kappa(x,y)+\frac{\alpha}{\alpha+\beta} \frac{d'(x',y')}{d(x,y) + d'(x',y')} \kappa'(x',y').
\end{equation}
In order to obtain the lower bound of $\kappa_{(\alpha,\beta)}(\mathbf{x},\mathbf{y})$,
we estimate the Wasserstein distance from above.
Proposition \ref{prop:jointly convex} and Lemma \ref{lem:product prob meas} yield that
for every $\epsilon \in (0,1]$,
\begin{align}\label{eq:product Ricci1}
W(\nu_{(\alpha,\beta),\mathbf{x}}^{\varepsilon}, \nu_{(\alpha,\beta),\mathbf{y}}^{\varepsilon})
& = W \left(\frac{\beta}{\alpha+\beta}\nu^{\epsilon}_{x|x'}+\frac{\alpha}{\alpha+\beta}(\nu^{\epsilon}_{x'|x})', \frac{\beta}{\alpha+\beta}\nu^{\epsilon}_{y|y'}+\frac{\alpha}{\alpha+\beta}(\nu^{\epsilon}_{y'|y} )'\right) \\ \notag
& \leq \frac{\beta}{\alpha+\beta}W(\nu^{\epsilon}_{x|x'},\nu^{\epsilon}_{y|y'}) +\frac{\alpha}{\alpha+\beta}W((\nu^{\epsilon}_{x'|x})',(\nu^{\epsilon}_{y'|y})'),
\end{align}
where the probability measures $\nu^{\epsilon}_{x|x'},\nu^{\epsilon}_{y|y'},(\nu^{\epsilon}_{x'|x})',(\nu^{\epsilon}_{y'|y})'$ are defined as (\ref{eq:product sub prob}).
Let us estimate $W(\nu^{\epsilon}_{x|x'},\nu^{\epsilon}_{y|y'})$ from above.
To do so,
we fix an optimal coupling $\pi_{xy}:V\times V\to [0,1]$ of $(\nu_x^{\varepsilon}, \nu_y^{\varepsilon})$,
and define a probability measure $\pi:(V\times V')\times (V\times V')\to [0,1]$ by
\begin{equation*}
\pi(\mathbf{z},\mathbf{w}):=\pi_{xy}(z,w)\delta_{x'}(z')\delta_{y'}(w').
\end{equation*}
We can check that
$\pi$ is a coupling of $(\nu^{\epsilon}_{x|x'},\nu^{\epsilon}_{y|y'})$.
From (\ref{eq:product distance}) we deduce
\begin{align}\label{eq:product Ricci2}
W(\nu^{\epsilon}_{x|x'},\nu^{\epsilon}_{y|y'})&\leq \sum_{\mathbf{z},\mathbf{w}\in V\times V'} d_{(\alpha,\beta)}(\mathbf{z},\mathbf{w})\pi(\mathbf{z},\mathbf{w}) \\ \notag
                                                                                                    & = \sum_{z,w\in V}\sum_{z',w'\in V'} (d(z,w) + d'(z',w'))\pi_{xy}(z,w)\delta_{x'}(z')\delta_{y'}(w') \\ \notag
                                                                                                    & = \sum_{z,w\in V} (d(z,w) + d'(x',y'))\pi_{xy}(z,w)= W(\nu_x^{\varepsilon}, \nu_y^{\varepsilon}) + d'(x', y'). 
\end{align}
We can also prove
\begin{equation}\label{eq:product Ricci3}
W((\nu^{\epsilon}_{x'|x})',(\nu^{\epsilon}_{y'|y})')\leq d(x,y)+W((\nu_{x'}^{\varepsilon})', (\nu_{y'}^{\varepsilon})')
\end{equation}
by considering a coupling $\pi':(V\times V')\times (V\times V')\to [0,1]$ of $((\nu^{\epsilon}_{x'|x})',(\nu^{\epsilon}_{y'|y})')$ defined as
\begin{equation*}
\pi'(\mathbf{z},\mathbf{w}):=\pi'_{x'y'}(z',w')\delta_{x}(z)\delta_{y}(w)
\end{equation*}
for a fixed optimal coupling $\pi'_{x'y'}:V'\times V'\to [0,1]$ of $((\nu_{x'}^{\varepsilon})',(\nu_{y'}^{\varepsilon})')$,
and by applying the same calculation as in (\ref{eq:product Ricci2}).
Substituting (\ref{eq:product Ricci2}) and (\ref{eq:product Ricci3}) into (\ref{eq:product Ricci1}),
we arrive at
\begin{align}\label{eq:pre product curvature}
W(\nu_{(\alpha,\beta),\mathbf{x}}^{\varepsilon}, \nu_{(\alpha,\beta),\mathbf{y}}^{\varepsilon})
&\leq \frac{\beta}{\alpha+\beta} \left( W(\nu_x^{\varepsilon}, \nu_y^{\varepsilon}) + d'(x',y') \right)\\ \notag
&\quad +\frac{\alpha}{\alpha+\beta}\left( d(x,y) + W((\nu_{x'}^{\varepsilon})', (\nu_{y'}^{\varepsilon})') \right).
\end{align}
Since $x\neq y$ and $x'\neq y'$,
we can rewrite (\ref{eq:pre product curvature}) as
\begin{align*}
1-\frac{W(\nu_{(\alpha,\beta),\mathbf{x}}^{\varepsilon}, \nu_{(\alpha,\beta)\mathbf{y}}^{\varepsilon})}{d_{(\alpha,\beta)}(\mathbf{x},\mathbf{y})}
& \geq \frac{\beta}{\alpha+\beta} \frac{d(x,y)}{d_{(\alpha,\beta)}(\mathbf{x},\mathbf{y})}  \left(  1-\frac{W(\nu_{x}^{\varepsilon}, \nu_{y}^{\varepsilon})}{d(x,y)}    \right)\\ \notag
&\quad +\frac{\alpha}{\alpha+\beta} \frac{d'(x',y')}{d_{(\alpha,\beta)}(\mathbf{x},\mathbf{y})}   \left(  1-\frac{W((\nu_{x'}^{\varepsilon})', (\nu_{y'}^{\varepsilon})')}{d'(x',y')}    \right).
\end{align*}
Dividing its both sides by $\epsilon$,
we obtain the relation for $\kappa_{\epsilon}$ defined (\ref{eq:pre Ricci curvature}).
Moreover,
letting $\epsilon \to 0$,
we conclude (\ref{eq:left product Ricci}) from (\ref{eq:product distance}).

We next prove the opposite inequality of (\ref{eq:left product Ricci}).
Take functions $f:V\to \mathbb{R}$ and $f':V'\to \mathbb{R}$ with $\nabla_{xy}f=1$ and $\nabla'_{x'y'}f'=1$,
where $\nabla_{xy}$ and $\nabla'_{x'y'}$ are the gradient operators over $(V,\mu)$ and $(V',\mu')$ defined as (\ref{eq:gradient operator}),
respectively.
We define a function $\mathbf{f}:V\times V'\to \mathbb{R}$ by
\begin{equation*}
\mathbf{f}(\mathbf{x}):=f(x)+f'(x').
\end{equation*}
Then we deduce $\nabla_{(\alpha,\beta),\mathbf{x}\mathbf{y}}\mathbf{f}=1$ from (\ref{eq:product distance}),
where $\nabla_{(\alpha,\beta),\mathbf{x}\mathbf{y}}$ denotes the gradient operator over $(V, \mu) \square_{(\alpha,\beta)} (V', \mu')$.
Hence,
by Theorem \ref{thm:limit free formula} and Lemma \ref{lem:product Laplacian},
we have
\begin{align*}
\kappa_{(\alpha,\beta)}(\mathbf{x},\mathbf{y})
&\leq \nabla_{(\alpha,\beta),\mathbf{x}\mathbf{y}}\mathcal{L}_{(\alpha,\beta)}\mathbf{f}
=\frac{\mathcal{L}_{(\alpha,\beta)}\mathbf{f}(\mathbf{y})-\mathcal{L}_{(\alpha,\beta)}\mathbf{f}(\mathbf{x})}{d_{(\alpha,\beta)}(\mathbf{x},\mathbf{y})}\\
&=\frac{1}{d_{(\alpha,\beta)}(\mathbf{x},\mathbf{y})}\left( \frac{\beta}{\alpha+\beta} (\mathcal{L}f(y)-\mathcal{L}f(x))+\frac{\alpha}{\alpha+\beta} (\mathcal{L}'f'(y')-\mathcal{L}'f'(x'))   \right)\\
&=\frac{\beta}{\alpha+\beta} \frac{d(x,y)}{d_{(\alpha,\beta)}(\mathbf{x},\mathbf{y})} \nabla_{xy}\mathcal{L}f+\frac{\alpha}{\alpha+\beta} \frac{d'(x',y')}{d_{(\alpha,\beta)}(\mathbf{x},\mathbf{y})}\nabla'_{x'y'} \mathcal{L}'f'.
\end{align*}
Since the functions $f$ and $f'$ are arbitrary,
we conclude the opposite one by using Theorem \ref{thm:limit free formula} again.
Thus we complete the proof.
\end{proof}

\begin{rem}
Lin-Lu-Yau \cite{LLY} have proved (\ref{eq:pre product Ricci1}) and (\ref{eq:pre product Ricci2}) when
$(V,\mu)$ is undirected $r$-regular,
$(V',\mu')$ is undirected $r'$-regular,
$\alpha=1/r,\beta=1/r'$ and $\mathbf{x}\rightarrow \mathbf{y}$ (see Theorem 3.1 in \cite{LLY}).
Our method of the proof of (\ref{eq:left product Ricci}) is based on that of Theorem 3.1 in \cite{LLY} (see Claim 1 in \cite{LLY}).
On the other hand,
for the proof of the opposite one,
their method seems not to work in our setting due to the lack of symmetry of the distance function (see Claim 2 in \cite{LLY}).
Our method via Theorem \ref{thm:limit free formula} seems to be new and more clear.
\end{rem}

The formulas (\ref{eq:product Ricci}), (\ref{eq:pre product Ricci1}), (\ref{eq:pre product Ricci2}) also hold for simple, connected, locally finite, infinite weighted undirected graphs,
whose proofs are identical.

\section{Estimates}\label{sec:Estimates of Ricci curvature}
In the present subsection,
we discuss several upper and lower bounds of our Ricci curvature.

\subsection{Lower bounds}
For $x,y\in V$ with $x\neq y$,
we set
\begin{equation*}
\mathcal{D}(x,y):=\max\left\{d(x,y),d(y,x)\right\}.
\end{equation*}
We first study a lower bound of our Ricci curvature (cf. Theorems 2 and 5 in \cite{JL}).
\begin{prop}\label{prop:basic lower Ricci bound}
For $x,y\in V$ with $x\neq y$,
we have
\begin{align}\label{eq:pre basic lower Ricci bound}
\kappa(x,y)\geq &-\frac{2 \mathcal{D}(x,y)}{d(x,y)}\,\left(1-\mathcal{P}(x,y)-\mathcal{P}(y,x) \right)_{+}+\frac{1}{d(x,y)}\left( d(x,y)+\mathcal{D}(x,y)-\mathcal{H}(y,x)  \right)\\ \notag
                        &-\frac{\mathcal{D}(x,y)-d(y,x)}{d(x,y)}\,\left(\mathcal{P}(x,y)+\mathcal{P}(y,x)\right),
\end{align}
where $(\cdot)_{+}$ denotes its positive part.
Moreover,
if $(x,y)\in E$,
then we have
\begin{equation}\label{eq:basic lower Ricci bound}
\kappa(x,y)\geq -2d(y,x)\,\left( 1-\mathcal{P}(x,y)-\mathcal{P}(y,x)  \right)_{+}+\left(  1+d(y,x)-\mathcal{H}(y,x)  \right).
\end{equation}
\end{prop}
\begin{proof}
Jost-Liu \cite{JL} shown (\ref{eq:basic lower Ricci bound}) in the undirected case,
whose primitive version has been established by Lin-Yau \cite{LY} (see Theorem 5 in \cite{JL}, and also Proposition 1.5 in \cite{LY}, Theorem 2 in \cite{JL}).
We will calculate along the line of the proof of Theorem 2 in \cite{JL}.

In view of Lemma \ref{lem:preconcavity},
it suffices to prove that
$\kappa_{1}(x,y)$ is bounded from below by the right hand side of (\ref{eq:basic lower Ricci bound}).
To do so,
let us estimate $W(\nu^{1}_{x},\nu^{1}_{y})$ from above.
Note that
$\nu^{1}_{x}(z)=\mathcal{P}(x,z)$ and $\nu^{1}_{y}(z)=\mathcal{P}(y,z)$ for all $z\in V$.
From Proposition \ref{prop:Kantorovich duality}
we deduce
\begin{align*}
W(\nu^{1}_{x},\nu^{1}_{y})&= \sup_{f\in \Lip_1(V)} \sum_{z\in V} f(z)(\nu_{y}^{1}(z)-\nu_{x}^{1}(z))\\
                                         &= \sup_{f\in \Lip_1(V)}\Biggl\{ \Biggl( \sum_{z\in V\setminus \{x\}}   \left(  f(z)-f(y)  \right)\mathcal{P}(y,z)     \Biggl)-\Biggl( \sum_{z\in V\setminus \{y\}}   \left(  f(z)-f(x)  \right)\mathcal{P}(x,z)     \Biggl)\\
                                         &\qquad \qquad \qquad\qquad \qquad \qquad \qquad \qquad +\left(f(y)-f(x)\right)\left(1-\mathcal{P}(x,y)-\mathcal{P}(y,x)\right) \Biggl\}.
\end{align*}
For any $f\in \Lip_1(V)$
we have
\begin{equation*}
f(z)-f(y)\leq d(y,z),\quad f(z)-f(x)\geq -d(z,x),\quad \vert f(y)-f(x) \vert\leq \mathcal{D}(x,y),
\end{equation*}
and hence
\begin{align*}
W(\nu^{1}_{x},\nu^{1}_{y})&\leq \sum_{z\in V\setminus \{x\}}   d(y,z)\mathcal{P}(y,z)+\sum_{z\in V\setminus \{y\}}   d(z,x)\mathcal{P}(x,z)\\
                                         &\qquad \qquad \qquad\qquad \qquad \,\,  +\mathcal{D}(x,y)\,\left\vert 1-\mathcal{P}(x,y)-\mathcal{P}(y,x)  \right\vert\\
                                         &=\left(-\mathcal{H}_{y}-d(y,x)\mathcal{P}(y,x)\right)+\left(  -\olarrow{\mathcal{H}}_{x} -d(y,x)\mathcal{P}(x,y) \right)\\
                                         &\quad +\mathcal{D}(x,y)\,  \left( 2 \left(1-\mathcal{P}(x,y)-\mathcal{P}(y,x) \right)_{+}  -\left(  1-\mathcal{P}(x,y)-\mathcal{P}(y,x)   \right) \right)\\
                                         &=\mathcal{H}(y,x)-d(y,x)\,\left(   \mathcal{P}(x,y)+\mathcal{P}(y,x) \right)\\
                                         &\quad +\mathcal{D}(x,y)\,  \left( 2 \left(1-\mathcal{P}(x,y)-\mathcal{P}(y,x) \right)_{+}  -\left(  1-\mathcal{P}(x,y)-\mathcal{P}(y,x)   \right) \right)\\
                                         &=2 \mathcal{D}(x,y)\,\left(1-\mathcal{P}(x,y)-\mathcal{P}(y,x) \right)_{+}-\left( \mathcal{D}(x,y)-\mathcal{H}(y,x)  \right)\\
                                         &\quad +\left(\mathcal{D}(x,y)-d(y,x)\right)\,\left(\mathcal{P}(x,y)+\mathcal{P}(y,x)\right),
\end{align*}
here we used (\ref{eq:mean curvature formula}), (\ref{eq:reverse mean curvature formula}).
This proves (\ref{eq:pre basic lower Ricci bound}).

When $(x,y)\in E$,
we have $d(x,y)=1$.
Furthermore,
\begin{equation*}
\mathcal{D}(x,y)=\max\left\{d(x,y),d(y,x)\right\}=\max\left\{1,d(y,x)\right\}=d(y,x)
\end{equation*}
since $d(y,x)\geq 1$.
Substituting these equalities into (\ref{eq:pre basic lower Ricci bound}),
we see that
the first term in the right hand of (\ref{eq:pre basic lower Ricci bound}) becomes that of (\ref{eq:basic lower Ricci bound}),
the second term also does,
and the third term vanishes.
Thus we obtain (\ref{eq:basic lower Ricci bound}).
\end{proof}

\begin{rem}\label{rem:unweighted basic lower Ricci bound}
In the undirected case,
under the same setting as in Proposition \ref{prop:basic lower Ricci bound},
we have $d(y,x)=1$ and $\mathcal{H}(x,y)=2$.
In particular,
the second term of the right hand side of (\ref{eq:basic lower Ricci bound}) vanishes,
and hence its right hand side coincides with that of Theorem 5 in \cite{JL}.
\end{rem}

For regular graphs,
we derive an another lower bound in terms of the inscribed radius instead of the asymptotic mean curvature (cf. Theorem 3 in \cite{JL}).
\begin{prop}\label{prop:lower bound of Ricci curvature}
For $r\geq 1$,
let $(V,\mu)$ be an $r$-regular graph.
Then for all edge $(x,y)\in E$,
\begin{equation*}
\kappa(x,y) \geq \frac{1-r}{2r} - \sum_{z \in N_{x} \setminus \olarrow{N}_{y}} \frac{\IR_{z}V}{2r},
\end{equation*}
where $N_{x}$ and $\olarrow{N}_{y}$ are defined as $(\ref{eq:neighborhoods})$,
and $\IR_{z} V$ is done as $(\ref{eq:inscribed radius})$.
\end{prop}
\begin{proof}
We take a coupling between $\nu^{\epsilon}_{x}$ and $\nu^{\epsilon}_{y}$.
Our transfer plan moving $\nu^{\epsilon}_{x}$ to $\nu^{\epsilon}_{y}$ should be as follows:
	\begin{enumerate}
	\item Move the mass of $1- \epsilon$ from $x$ to $y$. The distance is $1$;
	\item Move the mass of $\epsilon /2 r$ from $y$ to a fixed $y_{0} \in N_{y}$, and move the mass of $\epsilon / 2r$ from a fixed $x_{0} \in \olarrow{N}_{x}$ to $x$. These distances are $1$;
	\item Move the mass of $\epsilon / 2r$ to itself at $N_{x}  \cap \olarrow{N}_{y}$;
	\item Move the mass of $\epsilon /2 r$ from a vertex in $ \olarrow{N}_{x} \setminus \left\{x_{0} \right\}$ to a vertex in $N_{y} \setminus \left\{ y_{0} \right\}$. The distance is at most $3$;
	\item Move the mass of $\epsilon /2 r$ from a vertex $z$ in $N_{x} \setminus \olarrow{N}_{y}$ to a vertex in $\olarrow{N}_{y} \setminus N_{x}$. The distance is at most $\IR_z V$.
	\end{enumerate}
	By this transfer plan, calculating the Wasserstein distance from $\nu^{\epsilon}_{x}$ to $\nu^{\epsilon}_{y}$, we have
\begin{align*}
W(\nu^{\epsilon}_{x}, \nu^{\epsilon}_{y}) &\leq (1 - \epsilon) \times 1 + \frac{\epsilon}{2 r} \times 1 + \frac{\epsilon}{2 r} \times 1 + (r -1) \frac{\epsilon}{2r} \times 3 \\
&\quad + \sum_{z \in N_{x} \setminus \olarrow{N}_{y}} \frac{\epsilon}{2r} \times \IR_z V\\
&= 1 - \epsilon \left( \frac{1- r}{2r}  - \sum_{z \in N_{x} \setminus \olarrow{N}_{y}} \cfrac{\IR_{z}V}{2r} \right).
\end{align*}
This completes the proof.
\end{proof}

\subsection{Upper bounds}
We next examine an upper bound (cf. Theorems 4 and 7 in \cite{JL}).
\begin{prop}\label{prop:upper bound of Ricci curvature}
For every edge $(x,y)\in E$
we have
\begin{equation}\label{eq:upper bound of Ricci curvature}
\kappa (x,y) \leq \mathcal{P}(x,y)+ \mathcal{P}(y,x)+ \left(\sum_{z \in \mathcal{N}_{x} \cap \mathcal{N}_{y}}\mathcal{P}(x,z) \right) \wedge   \left( \sum_{z \in \mathcal{N}_{x} \cap \mathcal{N}_{y}}\mathcal{P}(y,z) \right),
\end{equation}
where $s \wedge t := \min \{s,t\}$.
\end{prop}
\begin{proof}
We show the desired inequality by modifying the proof of Theorem 4 in \cite{JL}.
Take a coupling $\pi$ of $(\nu^{\epsilon}_{x},\nu^{\epsilon}_{y})$.
Note that
$\nu^{\epsilon}_{x}$ is supported on $\{x\}\cup \mathcal{N}_{x}$.
Hence
$\pi$ is supported on
\begin{equation*}
\mathcal{N}_{\pi}:=\left(\mathcal{N}_{xy}\times \mathcal{N}_{xy}\right) \setminus \left( \left(   \left(  \mathcal{N}_{y}\setminus (\{x\}\cup\mathcal{N}_{x})   \right) \times \mathcal{N}_{xy} \right)     \cup  \left(\mathcal{N}_{xy}\times  \left(\mathcal{N}_{x}\setminus (\{y\}\cup\mathcal{N}_{y})\right)  \right)   \right),
\end{equation*}
where we set $\mathcal{N}_{xy}:=\mathcal{N}_{x}\cup \mathcal{N}_{y}$.
Let us define a subset $\mathcal{N}_{\pi,0}$ of $\mathcal{N}_{\pi}$ as
\begin{equation*}
\mathcal{N}_{\pi,0}:=\mathcal{N}_{\pi} \setminus \left(\{(x,x),(y,y)\}\cup \left( (\mathcal{N}_{x}\cap \mathcal{N}_{y})\times (\mathcal{N}_{x}\cap \mathcal{N}_{y})\right) \right).
\end{equation*}
Notice that
$d(z,w)\geq 1$ for all $(z,w)\in \mathcal{N}_{\pi,0}$.
It follows that
\begin{align}\label{eq:upper1}
\sum_{z,w\in V}d(z,w)\pi(z,w)&\geq \sum_{(z,w)\in \mathcal{N}_{\pi,0}}\pi(z,w)\\ \notag
                                              &   =   1-\left(\pi(x,x)+\pi(y,y)+\sum_{(z,w)\in (\mathcal{N}_{x}\cap \mathcal{N}_{y})\times (\mathcal{N}_{x}\cap \mathcal{N}_{y})} \pi(z,w)   \right).
\end{align}

By $\pi \in \Pi(\nu^{\epsilon}_{x},\nu^{\epsilon}_{y})$, we have
\begin{equation}\label{eq:upper2}
\pi(x,x)\leq \sum_{z\in V} \pi(z,x)=\nu^{\epsilon}_{y}(x)=\epsilon \mathcal{P}(y,x),\quad \pi(y,y)\leq \sum_{w\in V} \pi(y,w)=\nu^{\epsilon}_{x}(y)=\epsilon \mathcal{P}(x,y).
\end{equation}
Further,
for every $z\in \mathcal{N}_{x}\cap \mathcal{N}_{y}$ we see
\begin{equation*}
\sum_{w\in \mathcal{N}_{x}\cap \mathcal{N}_{y}}\pi(z,w)\leq \sum_{w\in V}\pi(z,w) =\nu^{\epsilon}_{x}(z)=\epsilon \mathcal{P}(x,z),
\end{equation*}
and hence
\begin{equation}\label{eq:upper3}
\sum_{(z,w)\in (\mathcal{N}_{x}\cap \mathcal{N}_{y})\times (\mathcal{N}_{x}\cap \mathcal{N}_{y})} \pi(z,w)  \leq \epsilon \sum_{z\in \mathcal{N}_{x}\cap \mathcal{N}_{y}}\mathcal{P}(x,z).
\end{equation}
Similarly,
for every $w\in \mathcal{N}_{x}\cap \mathcal{N}_{y}$,
\begin{equation*}
\sum_{z\in \mathcal{N}_{x}\cap \mathcal{N}_{y}}\pi(z,w)\leq \sum_{z\in V}\pi(z,w) =\nu^{\epsilon}_{y}(w)=\epsilon \mathcal{P}(y,w),
\end{equation*}
and thus
\begin{equation}\label{eq:upper4}
\sum_{(z,w)\in (\mathcal{N}_{x}\cap \mathcal{N}_{y})\times (\mathcal{N}_{x}\cap \mathcal{N}_{y})} \pi(z,w)  \leq \epsilon \sum_{w\in \mathcal{N}_{x}\cap \mathcal{N}_{y}}\mathcal{P}(y,w).
\end{equation}
We now combine (\ref{eq:upper1}), (\ref{eq:upper2}), (\ref{eq:upper3}), (\ref{eq:upper4}).
Since $\pi$ is arbitrary,
we conclude
\begin{equation*}
W(\nu^{\epsilon}_{x},\nu^{\epsilon}_{y}) \geq 1-\epsilon \,\left( \mathcal{P}(x,y)+\mathcal{P}(y,x)+  \left(\sum_{z \in \mathcal{N}_{x} \cap \mathcal{N}_{y}}\mathcal{P}(x,z) \right) \wedge   \left( \sum_{z \in \mathcal{N}_{x} \cap \mathcal{N}_{y}}\mathcal{P}(y,z) \right)  \right).
\end{equation*}
By $d(x,y)=1$ and by the definition of $\kappa(x,y)$,
we arrive at the desired one.
\end{proof}

\begin{rem}\label{rem:simpler estimate}
We observe that
the right hand side of (\ref{eq:upper bound of Ricci curvature}) is at most
\begin{equation*}
\mathcal{P}(x,y)+ \mathcal{P}(y,x)+ \sum_{z \in \mathcal{N}_{x} \cap \mathcal{N}_{y}} \mathcal{P}(x,z).
\end{equation*}
that is smaller than or equal to $1+\mathcal{P}(y,x)$.
Therefore
we can conclude a simpler estimate $\kappa (x,y)\leq 1+\mathcal{P}(y,x)$.
\end{rem}

\section{Curvature-dimension conditions}\label{sec:Curvature-dimension conditions}
The aim of this section is to study the relation between our Ricci curvature and the curvature-dimension inequalities of Bakry-\'Emery type.

\subsection{Curvature-dimension inequalities}
Let us recall the notion of \textit{$\Gamma$-operator} (or \textit{carr\'e du champ}), 
and the \textit{$\Gamma_{2}$-operator} (or \textit{carr\'e du champ it\'er\'e}) of Bakry-\'Emery \cite{BE} to formulate the curvature-dimension inequality. 
The \textit{$\Gamma$-operator},
and the \textit{$\Gamma_{2}$-operator} for the (negative) Chung Laplacian $\Delta$ are defined as follows (see \cite{BE}, and also Chapter 14 in \cite{V2}):
\begin{align*}
\Gamma(f_{0},f_{1})&:=\frac{1}{2}\left(\Delta(f_{0}\, f_{1})-f_{0}\, \Delta f_{1}-f_{1}\, \Delta f_{0} \right),\\
\Gamma_{2}(f_{0},f_{1})&:=\frac{1}{2}\left(\Delta \Gamma(f_{0},f_{1})-\Gamma(f_{0},\Delta f_{1})-\Gamma(f_{1},\Delta f_{0}) \right)
\end{align*}
for functions $f_{0},f_{1}:V\to \mathbb{R}$.

For a function $f:V\to \mathbb{R}$,
we define a function $\mathcal{G}f:V\to \mathbb{R}$ by
\begin{equation*}\label{eq:additional term}
\mathcal{G}f(x):=\frac{1}{4}\sum_{y,z\in V} \left(f(x)-2f(y)+f(z)    \right)^{2}\mathcal{P}(x,y)\mathcal{P}(y,z).
\end{equation*}
We begin with the following formulas:
\begin{prop}\label{prop:Gamma formula}
For all $f:V\to \mathbb{R}$ we have
\begin{align}\label{eq:gradient}
\Gamma(f,f)(x)
&=\frac{1}{2}\sum_{y\in V} (f(x)-f(y))^{2} \mathcal{P}(x,y),\\ \notag
\Delta \Gamma(f,f)(x)
&=\frac{1}{2}\sum_{y,z\in V} \left(f(x)-2f(y)+f(z)    \right)^{2}\mathcal{P}(x,y)\mathcal{P}(y,z)\\ \notag
&\quad -\left(\sum_{y\in V}\left(f(x)-f(y)   \right) \mathcal{P}(x,y)\right)\left(\sum_{z\in V}\left( f(x)-2f(y)+f(z)    \right)\mathcal{P}(y,z)\right),\\ \notag
2\Gamma(f,\Delta f)(x)
&=-\left(\Delta f(x)\right)^{2}-\left(\sum_{y\in V}(f(x)-f(y))\mathcal{P}(x,y)\right)\left(\sum_{z\in V}(f(z)-f(y))\mathcal{P}(y,z)\right).
\end{align}
In particular,
\begin{equation}\label{eq:curvature dimension1}
\Gamma_{2}(f,f)=\mathcal{G}f-\Gamma(f,f)+\frac{1}{2}(\Delta f)^{2}.
\end{equation}
\end{prop}
\begin{proof}
We recall that
the Perron measure $\mathfrak{m}$ and the value $\mathfrak{m}_{xy}$ are defined as (\ref{eq:Perron Frobenius}) and (\ref{eq:symmetric edge weight}),
respectively.
Keeping in mind $\mathcal{P}(x,y)=\mathfrak{m}_{xy}/\mathfrak{m}(x)$,
we can show the desired formulas from the same calculation as that done by Lin-Yau \cite{LY} (see Lemmas 1.4, 2.1 and (2.2) in \cite{LY}, and also Subsection 2.2 in \cite{JL}).
The calculation is left to the readers.
\end{proof}

We define the \textit{triangle function} $\mathcal{T}:V\to \mathbb{R}$ as follows (cf. Subsection 3.1 in \cite{JL}):
\begin{equation*}
\mathcal{T}(x):=\inf_{y\in \mathcal{N}_{x}}\vert \mathcal{N}_{x}\cap\mathcal{N}_{y} \vert,
\end{equation*}
where $\vert \cdot \vert$ denotes its cardinality.

Based on Proposition \ref{prop:Gamma formula},
we formulate the following curvature-dimension inequality:
\begin{thm}\label{thm:cd ineq}
For all $f:V\to \mathbb{R}$,
we have
\begin{equation*}
\Gamma_{2}(f,f)\geq \cfrac{1}{2}(\Delta f)^{2} + \mathcal{K}\,\Gamma(f,f),
\end{equation*}
where a function $\mathcal{K}:V\to \mathbb{R}$ is defined as
\begin{equation*}
\mathcal{K}(x):=\left(\inf_{y\in \mathcal{N}_{x}}\,\mathcal{P}(y,x)\right)\left\{2+\frac{\mathcal{T}(x)}{2}\left(\inf_{y\in \mathcal{N}_{x}} \inf_{z\in \mathcal{N}_{x}\cap \mathcal{N}_{y}}\, \frac{\mathcal{P}(y,z)}{\mathcal{P}(y,x)}\right)   \right\}-1.
\end{equation*}
\end{thm}
\begin{proof}
Jost-Liu \cite{JL} have proved a similar curvature-dimension inequality in the undirected case (cf. Theorems 9 and 10 in \cite{JL}).
We show the desired inequality along the line of the proof of Theorem 9 in \cite{JL}.
In view of (\ref{eq:curvature dimension1}),
it suffices to show that
$\mathcal{G}f/\Gamma(f,f)$ is bounded from below by $\mathcal{K}+1$.
From (\ref{eq:gradient}) we deduce
\begin{align}\label{eq:cd ineq1}
\mathcal{G}f(x)&=\frac{1}{4}\sum_{y,z\in V} \left(f(x)-2f(y)+f(z)    \right)^{2}\mathcal{P}(x,y)\mathcal{P}(y,z)\\ \notag
&= \sum_{y\in V}(f(x)-f(y))^{2} \mathcal{P}(x,y)\mathcal{P}(y,x)\\ \notag
&\quad + \cfrac{1}{4}\sum_{y\in V}\sum_{z \in V\setminus \{x\}}  \left(f(x)-2f(y)+f(z)    \right)^{2}\mathcal{P}(x,y)\mathcal{P}(y,z)\\ \notag
&\geq 2\mathcal{K}_{0}(x) \Gamma(f,f)(x)+\cfrac{1}{4}\sum_{y\in \mathcal{N}_{x}}\sum_{z \in \mathcal{N}_{y}\setminus \{x\}}\left(f(x)-2f(y)+f(z)    \right)^{2}\mathcal{P}(x,y)\mathcal{P}(y,z)\\ \notag
&\geq 2\mathcal{K}_{0}(x) \Gamma(f,f)(x)+\cfrac{1}{4}\sum_{y\in \mathcal{N}_{x}}\sum_{z \in \mathcal{N}_{x} \cap\mathcal{N}_{y}}\left(f(x)-2f(y)+f(z)    \right)^{2}\mathcal{P}(x,y)\mathcal{P}(y,z),
\end{align}
where a function $\mathcal{K}_{0}:V\to \mathbb{R}$ is defined as
\begin{equation*}
\mathcal{K}_{0}(x):=\inf_{y\in \mathcal{N}_{x}}\,\mathcal{P}(y,x).
\end{equation*}
We estimate the second term of the right hand side of (\ref{eq:cd ineq1}).
Define $\mathcal{G}_{0}f:V\to \mathbb{R}$ by
\begin{equation*}
\mathcal{G}_{0}f(x):=\sum_{y\in \mathcal{N}_{x}}\sum_{z \in \mathcal{N}_{x} \cap\mathcal{N}_{y}}\left(f(x)-2f(y)+f(z)    \right)^{2}\mathcal{P}(x,y)\mathcal{P}(y,z).
\end{equation*}
By using $\mathcal{P}(x,y)=\mathfrak{m}_{xy}/\mathfrak{m}(x)$,
we rewrite $\mathcal{G}_{0}f$ as
\begin{align*}
\mathcal{G}_{0}f(x)&=\frac{1}{\mathfrak{m}(x)}\sum_{y\in \mathcal{N}_{x}}  \sum_{z \in \mathcal{N}_{x} \cap\mathcal{N}_{y}}  \frac{\mathfrak{m}_{xy}}{\mathfrak{m}(y)}  \left(f(x)-2f(y)+f(z)    \right)^{2}\mathfrak{m}_{yz}\\ \notag
                               &=\frac{1}{\mathfrak{m}(x)}\sum_{y\in \mathcal{N}_{x}} \mathcal{P}(y,x) \sum_{z \in \mathcal{N}_{x} \cap\mathcal{N}_{y}}   \left(f(x)-2f(y)+f(z)    \right)^{2}\mathfrak{m}_{yz}.
\end{align*}
In particular,
\begin{equation}\label{eq:modified curvature dimension}
\mathcal{G}_{0}f(x)\geq \frac{\mathcal{K}_{0}(x)}{\mathfrak{m}(x)}\sum_{y\in \mathcal{N}_{x}}  \sum_{z \in \mathcal{N}_{x} \cap\mathcal{N}_{y}}   \left(f(x)-2f(y)+f(z)    \right)^{2}\mathfrak{m}_{yz}.
\end{equation}

We now observe that
for $y\in \mathcal{N}_{x}$ and $z\in \mathcal{N}_{x}\cap \mathcal{N}_{y}$,
\begin{align}\label{eq:triangle}
& \left(f(x)-2f(y)+f(z)    \right)^{2}\mathfrak{m}_{yz}+\left(f(x)-2f(z)+f(y)    \right)^{2}\mathfrak{m}_{zy}\\ \notag
& = \left(      \left(f(x)-2f(y)+f(z)    \right)^{2}+    \left(f(x)-2f(z)+f(y)    \right)^{2}  \right)\mathfrak{m}_{yz}\\ \notag
&=  \left(  \left( f(x)-f(y)    \right)^{2}+4\left( f(y)-f(z)    \right)^{2}+\left( f(x)-f(z)    \right)^{2}  \right)\mathfrak{m}_{yz}\\ \notag
&\geq  \left(  \left( f(x)-f(y)    \right)^{2}+\left( f(x)-f(z)    \right)^{2}  \right)\mathfrak{m}_{yz}\\ \notag
&\geq \mathcal{K}_{1}(x)\, \left(  (f(x)-f(y))^{2}\mathfrak{m}_{xy}+ (f(x)-f(z))^{2}\mathfrak{m}_{xz} \right),
\end{align}
where a function $\mathcal{K}_{1}:V\to \mathbb{R}$ is defined as
\begin{equation*}
\mathcal{K}_{1}(x):=\inf_{y\in \mathcal{N}_{x}} \inf_{z\in \mathcal{N}_{x}\cap \mathcal{N}_{y}}\, \frac{\mathfrak{m}_{yz}}{\mathfrak{m}_{yx}}=\inf_{y\in \mathcal{N}_{x}} \inf_{z\in \mathcal{N}_{x}\cap \mathcal{N}_{y}}\, \frac{\mathcal{P}(y,z)}{\mathcal{P}(y,x)}.
\end{equation*}
From (\ref{eq:triangle}),
and the \textit{triangle argument} which is the main idea of the proof of Theorem 9 in \cite{JL},
it follows that
the right hand side of (\ref{eq:modified curvature dimension}) is greater than or equal to
\begin{equation*}
\frac{\mathcal{K}_{0}(x) \,\mathcal{K}_{1}(x)}{\mathfrak{m}(x)}\sum_{y\in \mathcal{N}_{x}} \vert \mathcal{N}_{x}\cap \mathcal{N}_{y} \vert \,  \left(f(x)-f(y)    \right)^{2}\mathfrak{m}_{xy}.
\end{equation*}
Hence,
(\ref{eq:gradient}) leads us to
\begin{equation*}
\mathcal{G}_{0}f(x)\geq \mathcal{T}(x) \mathcal{K}_{0}(x) \,\mathcal{K}_{1}(x) \sum_{y\in \mathcal{N}_{x}} \left(f(x)-f(y)    \right)^{2}\mathcal{P}(x,y)=2\mathcal{T}(x)\mathcal{K}_{0}(x) \,\mathcal{K}_{1}(x)\Gamma(f,f)(x).
\end{equation*}
Therefore,
\begin{align*}
\mathcal{G}f(x)&\geq 2\mathcal{K}_{0}(x) \Gamma(f,f)(x)+\frac{1}{4}\mathcal{G}_{0}f(x)\\
                        &\geq \mathcal{K}_{0}(x)\left(2+\frac{1}{2}\mathcal{T}(x)\mathcal{K}_{1}(x)\right)\Gamma(f,f)(x)=\left(\mathcal{K}(x)+1\right)\Gamma(f,f)(x).
\end{align*}
This completes the proof.
\end{proof}

We can immediately derive the following simple one from Theorem \ref{thm:cd ineq}:
\begin{cor}\label{cor:simple curvature dimension}
For all $f:V\to \mathbb{R}$,
we have
\begin{equation*}
\Gamma_{2}(f,f)\geq \cfrac{1}{2}(\Delta f)^{2} + \widetilde{\mathcal{K}}\,\Gamma(f,f)
\end{equation*}
on $V$,
where a function $\widetilde{\mathcal{K}}:V\to \mathbb{R}$ is defined as
\begin{equation*}\label{eq:lower bound of curvature dimension}
\widetilde{\mathcal{K}}(x):=2\,\inf_{y\in \mathcal{N}_{x}}\mathcal{P}(y,x)-1.
\end{equation*}
\end{cor}

Lin-Yau \cite{LY} have established Corollary \ref{cor:simple curvature dimension} in the undirected case (see Theorems 1.2 and 1.3 in \cite{LY}, and also Theorem 8 in \cite{JL}).

\begin{rem}\label{rem:Yamada estimate}
In the unweighted case,
the third author \cite{Y2} has shown
\begin{equation*}
\Gamma_{2}(f,f)\geq \cfrac{1}{2}(\Delta f)^{2} + \mathcal{K}'\,\Gamma(f,f)
\end{equation*}
for a function $\mathcal{K}':V\to \mathbb{R}$ defined as
\begin{equation*}
\mathcal{K}'(x):=\min \left\{ \inf_{y\in N_{x}}P(y,x), \inf_{y\in \olarrow{N}_{x}} \olarrow{P}(y,x)  \right\}-1,
\end{equation*}
which is slightly weaker than Corollary \ref{cor:simple curvature dimension}.
\end{rem}

\subsection{Ricci curvature and curvature-dimension inequalities}

By Proposition \ref{prop:upper bound of Ricci curvature}
we obtain the following relation between our Ricci curvature and curvature-dimension inequality:
\begin{cor}
For $K\in \mathbb{R}$,
we assume $\inf_{x,y}\kappa(x,y)\geq K$,
where the infimum is taken over all $x,y\in V$ with $x \neq y$.
Then for every $f:V\to \mathbb{R}$
we have
\begin{equation*}
\Gamma_{2}(f,f) \geq \cfrac{1}{2}(\Delta f)^{2} + \widehat{\mathcal{K}} \Gamma(f,f)
\end{equation*}
on $V$,
where a function $\widehat{\mathcal{K}}:V\to \mathbb{R}$ is defined as
\begin{equation*}
\widehat{\mathcal{K}}(x):=2K-3+\frac{K-1}{2}\mathcal{T}(x)\left(\inf_{y\in \mathcal{N}_{x}} \inf_{z\in \mathcal{N}_{x}\cap \mathcal{N}_{y}}\, \frac{\mathcal{P}(y,z)}{\mathcal{P}(y,x)}\right).
\end{equation*}
\end{cor}
\begin{proof}
Let us fix $x\in V$.
For every $y\in N_{x}$,
Proposition \ref{prop:upper bound of Ricci curvature} implies $\kappa(x,y)\leq 1+\mathcal{P}(y,x)$,
and hence $\mathcal{P}(y,x)\geq K-1$ (see Remark \ref{rem:simpler estimate}).
Similarly,
in virtue of Proposition \ref{prop:upper bound of Ricci curvature},
we also see $\mathcal{P}(y,x)\geq K-1$ for all $y\in \olarrow{N}_{x}$.
Thus we obtain $\inf_{y\in \mathcal{N}_{x}}\mathcal{P}(y,x)\geq K-1$.
Due to Theorem \ref{thm:cd ineq},
we complete the proof.
\end{proof}

Also,
combining Proposition \ref{prop:upper bound of Ricci curvature} with Corollary \ref{cor:simple curvature dimension} implies:
\begin{cor}
For $K\in \mathbb{R}$,
we assume $\inf_{x,y}\kappa(x,y)\geq K$,
where the infimum is taken over all $x,y\in V$ with $x \neq y$.
Then for every $f:V\to \mathbb{R}$
we have
\begin{equation*}
\Gamma_{2}(f,f) \geq \cfrac{1}{2}(\Delta f)^{2} +  (2K-3)\Gamma(f,f).
\end{equation*}
\end{cor}

\section{Comparison geometric results}\label{sec:Comparison geometric results}
In the present section,
we study various comparison geometric results.

\subsection{Eigenvalue comparisons}\label{sec:Eigenvalue comparisons}
In this first subsection,
we study an eigenvalue comparison of Lichnerowicz type.
We denote by
\begin{equation*}
0=\lambda_{0}<\lambda_{1}\leq \lambda_{2} \leq \dots \leq \lambda_{n-1}
\end{equation*}
the eigenvalues of $\mathcal{L}$.
We here notice that
for any non-zero function $f:V\to \mathbb{R}$,
its associated \textit{Rayleigh quotient} is given by
\begin{equation*}
\mathcal{R}(f):=\frac{1}{2}\frac{\sum_{x,y\in V}(f(y)-f(x))^{2}\mathfrak{m}_{xy}}{(f,f)}
\end{equation*}
in view of Proposition \ref{prop:integration by parts}.

To derive an eigenvalue comparison,
for $\epsilon>0$,
we consider the \textit{$\epsilon$-averaging operator} $\mathcal{A}^{\epsilon}:\mathcal{F}\to \mathcal{F}$ defined as 
\begin{equation*}
\mathcal{A}^{\epsilon}f(x):=\sum_{z\in V}f(z)\nu^{\epsilon}_{x}(z),
\end{equation*}
where $\nu^{\epsilon}_{x}$ is defined as (\ref{eq:random walk}).
Let us verify the following:
\begin{lem}\label{lem:Lipschitz property of averaging operator}
For $L>0$,
let $f:V\to \mathbb{R}$ be an $L$-Lipschitz function.
For $K>0$,
we assume $\inf_{x,y\in V}\kappa_{\epsilon}(x,y)\geq \epsilon K$,
where $\kappa_{\epsilon}$ is defined as $(\ref{eq:pre Ricci curvature})$,
and the infimum is taken over all $x,y\in V$ with $x \neq y$.
Then $\mathcal{A}^{\epsilon}f$ is $(1-\epsilon\,K)L$-Lipschitz.
\end{lem}
\begin{proof}
Let $x,y\in V$.
Using Proposition \ref{prop:Kantorovich duality},
we have
\begin{align*}
\mathcal{A}^{\epsilon}f(y)-\mathcal{A}^{\epsilon}f(x) &  =   \sum_{z\in V}f(z)(\nu^{\epsilon}_{y}(z)-\nu^{\epsilon}_{x}(z))\\
                                                    &\leq L\, W(\nu^{\epsilon}_{x},\nu^{\epsilon}_{y})=(1- \,\kappa_{\epsilon}(x,y))\,L\,d(x,y)\leq (1-\epsilon K)\,L\,d(x,y).
\end{align*}
This proves the lemma.
\end{proof}

From Lemma \ref{lem:Lipschitz property of averaging operator},
we conclude the following eigenvalue comparison of Lichnerowicz type:
\begin{thm}\label{thm:LLY comparison}
For $K>0$,
we assume $\inf_{x,y}\kappa(x,y)\geq K$,
where the infimum is taken over all $x,y\in V$ with $x \neq y$.
Then we have
\begin{equation*}
\lambda_{1}\geq K.
\end{equation*}
\end{thm}
\begin{proof}
This estimate has been obtained by Lin-Lu-Yau \cite{LLY} in the undirected case (see Theorem 4.2 in \cite{LLY}, and cf. Proposition 30 in \cite{O} and Theorem 4 in \cite{BJL}).
One can show the desired inequality by the same argument as in the proof of Theorem 4.2 in \cite{LLY}, or Proposition 30 in \cite{O}.
We only give an outline of the proof.

Let $\mathcal{F}_{c}\subset \mathcal{F}$ denote the orthogonal complement of the set of all constant functions on $V$.
For sufficiently small $\epsilon>0$
the spectral radius $\SR(\mathcal{A}^{\epsilon})$ of $\mathcal{A}^{\epsilon}$ over $\mathcal{F}_{c}$ is equal to $1-\epsilon \lambda_{1}$.
On the other hand,
the spectral radius is known to be characterized as
\begin{equation*}
\SR(\mathcal{A}^{\epsilon})=\lim_{k\to \infty} \Vert (\mathcal{A}^{\epsilon})^{k} \Vert^{1/k}_{\op},
\end{equation*}
here $\Vert \cdot \Vert_{\op}$ is the operator norm induced from the norm
\begin{equation*}
\Vert f \Vert^{2}_{\Var}:=\frac{1}{2} \sum_{x,y\in V}(f(y)-f(x))^{2}\mathfrak{m}_{xy}
\end{equation*}
on $\mathcal{F}_{c}$.
In view of Lemma \ref{lem:Lipschitz property of averaging operator},
the same argument as in \cite{LLY}, \cite{O} leads to
\begin{equation*}
1-\epsilon \lambda_{1}=\SR(\mathcal{A}^{\epsilon})=\lim_{k\to \infty} \Vert (\mathcal{A}^{\epsilon})^{k} \Vert^{1/k}_{\op}\leq 1-\epsilon K.
\end{equation*}
This completes the proof.
\end{proof}

In the undirected case,
we have a further work on eigenvalue comparisons (see \cite{BJL}).

\subsection{Diameter comparisons}\label{sec:Diameter comparisons}
In this second subsection,
we examine a diameter comparison of Bonnet-Myers type.
Let us show the following:
\begin{thm}\label{thm:LLY diameter comparison}
Let $x,y \in V$ with $x \neq y$.
If $\kappa(x,y)>0$, then
\begin{equation*}
d(x,y) \leq \frac{\mathcal{H}(x,y)} {\kappa(x,y)}.
\end{equation*}
\end{thm}
\begin{proof}
We complete the proof by letting $\epsilon \to 0$ in (\ref{eq:preboundedness}).
\end{proof}
Lin-Lu-Yau \cite{LLY} have proved Theorem \ref{thm:LLY diameter comparison} in the undirected case (see Theorem 4.1 in \cite{LLY}).

Theorem \ref{thm:LLY diameter comparison} yields the following inscribed radius estimate:
\begin{cor}\label{cor:inrad comparison}
Let $x\in V$.
For $K>0$
we assume $\inf_{y\in V \setminus \{x\}}\kappa(x,y) \geq K$.
For $\Lambda \in [2,\infty)$
we further assume $\sup_{y\in V\setminus \{x\}}\mathcal{H}(x,y)\leq \Lambda$.
Then
\begin{equation*}
\IR_{x}V \leq \frac{\Lambda}{K}.
\end{equation*}
\end{cor}

In the undirected case,
we also have a further work on diameter comparisons (see \cite{CKKLMP}).

\subsection{Volume comparisons}\label{sec:Volume comparisons}
In this third subsection,
for $x\in V$ and $R\geq 0$,
we investigate volume comparisons for the \textit{(forward) metric sphere} and \textit{metric ball} defined as
\begin{equation*}
S_{R}(x):=\{y\in V \mid \rho_{x}(y)=R\},\quad B_{R}(x):=\{y\in V \mid \rho_{x}(y)\leq R\}.
\end{equation*}

To prove our volume comparisons,
we prepare the following lemma:
\begin{lem}\label{lem:vol3}
Let $x\in V$.
For $K\in \mathbb{R}$
we assume $\inf_{y\in V \setminus \{x\}}\kappa(x,y) \geq K$.
For $\Lambda \in [2,\infty)$
we further assume $\sup_{y\in V\setminus \{x\}}\mathcal{H}(x,y)\leq \Lambda$.
Then for all $R\geq 1$ with $KR\leq \Lambda$,
and $y\in S_{R}(x)$,
\begin{equation}\label{eq:vol3}
\sum_{z\in \mathcal{N}_{y}\cap S_{R+1}(x)}\mathcal{P}(y,z)\leq \frac{\Lambda-KR}{2}.
\end{equation}
\end{lem}
\begin{proof}
Take an optimal coupling $\pi$ of $(\nu^{\epsilon}_{x},\nu^{\epsilon}_{y})$.
We set
\begin{equation*}
\overline{\mathcal{N}}_{x}:=\{x\}\cup \mathcal{N}_{x},\quad \overline{\mathcal{N}}_{y}:=\{y\}\cup \mathcal{N}_{y}.
\end{equation*}
Note that
$\pi$ is supported on $\overline{\mathcal{N}}_{x} \times \overline{\mathcal{N}}_{y}$.
It holds that
\begin{align*}
(1-\kappa_{\epsilon}(x,y))R&=(1-\kappa_{\epsilon}(x,y))d(x,y)=W(\nu^{\epsilon}_{x},\nu^{\epsilon}_{y})=\sum_{z\in \overline{\mathcal{N}}_{x}}\sum_{w\in \overline{\mathcal{N}}_{y}}d(z,w)\pi(z,w)\\
                                &=\sum_{z\in \overline{\mathcal{N}}_{x}}\sum_{w\in \overline{\mathcal{N}}_{y}\cap S_{R+1}(x)}d(z,w)\pi(z,w)+\sum_{z\in \overline{\mathcal{N}}_{x}}\sum_{w\in \overline{\mathcal{N}}_{y}\setminus S_{R+1}(x)}d(z,w)\pi(z,w).
\end{align*}
For all $z\in \overline{\mathcal{N}}_{x}$ and $w\in \overline{\mathcal{N}}_{y}\cap S_{R+1}(x)$ we see
\begin{equation*}
d(z,w)\geq d(x,w)-d(x,z)= (R+1)-\rho_{x}(z)=R-\rho_{x}(z)+1.
\end{equation*}
For all $z\in \overline{\mathcal{N}}_{x}$ and $w\in \overline{\mathcal{N}}_{y}\setminus S_{R+1}(x)$
we also have
\begin{equation*}
d(z,w)\geq d(x,y)-d(x,z)-d(w,y)= R-\rho_{x}(z)-\olarrow{\rho}_{y}(w).
\end{equation*}
It follows that
\begin{align}\label{eq:keyvol1}
(1-\kappa_{\epsilon}(x,y))R&\geq \sum_{z\in \overline{\mathcal{N}}_{x}}\sum_{w\in \overline{\mathcal{N}}_{y}\cap S_{R+1}(x)}(R-\rho_{x}(z)+1)\pi(z,w)\\ \notag
                                             &\quad  +   \sum_{z\in \overline{\mathcal{N}}_{x}}\sum_{w\in \overline{\mathcal{N}}_{y}\setminus S_{R+1}(x)}(R-\rho_{x}(z)-\olarrow{\rho}_{y}(w))\pi(z,w)\\ \notag
                                             &=R+\sum_{z\in \overline{\mathcal{N}}_{x}}\sum_{w\in \overline{\mathcal{N}}_{y}}(-\rho_{x}(z)-\olarrow{\rho}_{y}(w))\pi(z,w)\\ \notag
                                             &\quad  +   \sum_{z\in \overline{\mathcal{N}}_{x}}\sum_{w\in \overline{\mathcal{N}}_{y}\cap S_{R+1}(x)}(1+\olarrow{\rho}_{y}(w))\pi(z,w).              
\end{align}
For the second term of the right hand side of (\ref{eq:keyvol1}),
we deduce
\begin{align}\label{eq:keyvol2}
&\quad\,\,  \sum_{z\in \overline{\mathcal{N}}_{x}}\sum_{w\in \overline{\mathcal{N}}_{y}}(-\rho_{x}(z)-\olarrow{\rho}_{y}(w))\pi(z,w)
=-\sum_{z\in \overline{\mathcal{N}}_{x}}  \rho_{x}(z)\nu^{\epsilon}_{x}(z)-\sum_{w\in \overline{\mathcal{N}}_{y}}  \olarrow{\rho}_{y}(z)\nu^{\epsilon}_{y}(z)\\ \notag
&=-\epsilon \,\sum_{z\in \mathcal{N}_{x}}  \rho_{x}(z)\mathcal{P}(x,z)-\epsilon\,\sum_{w\in \mathcal{N}_{y}}  \olarrow{\rho}_{y}(z)\mathcal{P}(y,z)
=\epsilon (\mathcal{H}_{x}+\olarrow{\mathcal{H}}_{y})=-\epsilon\,\mathcal{H}(x,y)
\end{align}
from (\ref{eq:mean curvature formula}) and (\ref{eq:reverse mean curvature formula}).
For the third term,
we also possess
\begin{align}\label{eq:keyvol3}
&\quad \,\, \sum_{z\in \overline{\mathcal{N}}_{x}}\sum_{w\in \overline{\mathcal{N}}_{y}\cap S_{R+1}(x)}(1+\olarrow{\rho}_{y}(w))\pi(z,w)
=\sum_{w\in \overline{\mathcal{N}}_{y}\cap S_{R+1}(x)}(1+\olarrow{\rho}_{y}(w))\nu^{\epsilon}_{y}(w)\\ \notag
&=\epsilon \sum_{w\in \mathcal{N}_{y}\cap S_{R+1}(x)}(1+\olarrow{\rho}_{y}(w))\mathcal{P}(y,w)\geq 2 \epsilon \sum_{w\in \mathcal{N}_{y}\cap S_{R+1}(x)}\mathcal{P}(y,w)
\end{align}
since $y$ does not belong to $\overline{\mathcal{N}}_{y}\cap S_{R+1}(x)$,
and $\olarrow{\rho}_{y}\geq 1$ on $\mathcal{N}_{y}$.
By (\ref{eq:keyvol1}), (\ref{eq:keyvol2}) and (\ref{eq:keyvol3}),
\begin{equation*}
(1-\kappa_{\epsilon}(x,y))R\geq R-\epsilon\,\mathcal{H}(x,y)+2 \epsilon \sum_{w\in \mathcal{N}_{y}\cap S_{R+1}(x)}\mathcal{P}(y,w).
\end{equation*}
Dividing the both sides of the above inequality by $\epsilon$,
and letting $\epsilon \to 0$,
we obtain
\begin{equation*}
\sum_{w\in \mathcal{N}_{y}\cap S_{R+1}(x)}\mathcal{P}(y,w)\leq \frac{1}{2} \left(\mathcal{H}(x,y)-\kappa(x,y)R\right)\leq \frac{\Lambda-KR}{2}.
\end{equation*}
This proves (\ref{eq:vol3}).
\end{proof}

We set
\begin{equation*}
\mathcal{M}:=\inf_{y\in V} \inf_{z\in \mathcal{N}_{y}} \mathcal{P}(z,y).
\end{equation*}
We establish the following volume comparison (cf. Theorem 1 and Corollary 3 in \cite{P}):
\begin{thm}\label{thm:volume comparison}
Let $x\in V$.
For $K\in \mathbb{R}$
we assume $\inf_{y\in V \setminus \{x\}}\kappa(x,y) \geq K$.
For $\Lambda \in [2,\infty)$
we further assume $\sup_{y\in V\setminus \{x\}}\mathcal{H}(x,y)\leq \Lambda$.
Then for every $R\geq 0$ with $KR\leq \Lambda$,
we have
\begin{equation}\label{eq:volume}
\frac{\mathfrak{m}(S_{R+1}(x))}{\mathfrak{m}(S_{R}(x))}\leq \frac{\Lambda-KR}{2\,\mathcal{M}},
\end{equation}
where $\mathfrak{m}(S_{R}(x))$ is defined as $(\ref{eq:measure})$.
\end{thm}
\begin{proof}
Paeng \cite{P} has obtained a similar result under a lower bound for $\kappa_{1}$ in the undirected and unweighted case (see Theorem 1 and Corollary 3 in \cite{P}).
We will prove (\ref{eq:volume}) along the line of the proof of Theorem 1 in \cite{P}.
We remark that if $K>0$, then $\IR_{x}V \leq \Lambda/K$ due to Corollary \ref{cor:inrad comparison}.

First,
we prove (\ref{eq:volume}) in the case of $R\geq 1$.
We have
\begin{equation}\label{eq:vol1}
\mathfrak{m}(S_{R+1}(x))\leq  \sum_{y\in S_{R}(x)}\mathfrak{m}( \mathcal{N}_{y}\cap S_{R+1}(x))= \sum_{y\in S_{R}(x)} \sum_{z\in \mathcal{N}_{y}\cap S_{R+1}(x)} \mathfrak{m}(z).
\end{equation}
Here we used
\begin{equation*}
S_{R+1}(x)=\bigcup_{y\in S_{R}(x)}\left( \mathcal{N}_{y}\cap S_{R+1}(x)  \right). 
\end{equation*}
On the other hand,
Lemma \ref{lem:vol3} together with $\mathcal{P}(y,z)=\mathfrak{m}_{yz}/\mathfrak{m}(y)$ leads us that
\begin{equation*}
\mathcal{M} \sum_{z\in \mathcal{N}_{y}\cap S_{R+1}(x)}\mathfrak{m}(z) \leq   \sum_{z\in \mathcal{N}_{y}\cap S_{R+1}(x)}\mathfrak{m}_{zy}=   \sum_{z\in \mathcal{N}_{y}\cap S_{R+1}(x)}\mathfrak{m}_{yz}\leq \frac{\Lambda-KR}{2}\mathfrak{m}(y),
\end{equation*}
and hence
\begin{equation}\label{eq:vol2}
\sum_{y\in S_{R}(x)} \sum_{z\in \mathcal{N}_{y}\cap S_{R+1}(x)}\mathfrak{m}(z)   \leq \frac{\Lambda-KR}{2\,\mathcal{M}}\mathfrak{m}(S_{R}(x)).
\end{equation}
Combining (\ref{eq:vol1}) and (\ref{eq:vol2}),
we arrive at the desired inequality (\ref{eq:volume}) when $R\geq 1$.

Next,
we consider the case of $R=0$.
The forward metric sphere $S_{1}(x)$ coincides with the outer neighborhood $N_{x}$,
and hence $S_{1}(x)$ is contained in $\mathcal{N}_{x}$.
It follows that
\begin{align*}
\mathcal{M}\,\mathfrak{m}(S_{1}(x))\leq \mathcal{M}\,\sum_{y\in \mathcal{N}_{x}}\mathfrak{m}(y)\leq \sum_{y\in \mathcal{N}_{x}}\mathfrak{m}_{yx}=\sum_{y\in \mathcal{N}_{x}}\mathfrak{m}_{xy}=\mathfrak{m}(x).
\end{align*}
By $\Lambda \in [2,\infty)$,
we see $\mathfrak{m}(S_{1}(x))/\mathfrak{m}(x)\leq \Lambda/(2\mathcal{M})$.
Thus we complete the proof.
\end{proof}

One can also conclude the following results by using Theorem \ref{thm:volume comparison} along the line of the proof of Theorem 1 in \cite{P}.
\begin{cor}
Let $x\in V$.
For $K\in \mathbb{R}$
we assume $\inf_{y\in V \setminus \{x\}}\kappa(x,y) \geq K$.
For $\Lambda \in [2,\infty)$
we further assume $\sup_{y\in V\setminus \{x\}}\mathcal{H}(x,y)\leq \Lambda$.
Then for every $R\geq 1$ with $(R-1)K\leq \Lambda$,
\begin{equation*}
\mathfrak{m}(S_{R}(x))\leq \mathfrak{m}(x) \prod^{R-1}_{i=0}  \left(\frac{\Lambda-i\,K}{2\,\mathcal{M}}\right).
\end{equation*}
\end{cor}

\begin{cor}\label{cor:Bishop}
Let $x\in V$.
For $K\in \mathbb{R}$
we assume $\inf_{y\in V \setminus \{x\}}\kappa(x,y) \geq K$.
For $\Lambda \in [2,\infty)$
we further assume $\sup_{y\in V\setminus \{x\}}\mathcal{H}(x,y)\leq \Lambda$.
Then for every $R\geq 1$ with $(R-1)K\leq \Lambda$,
\begin{equation*}
\mathfrak{m}(B_{R}(x))\leq \mathfrak{m}(x)\left(1+\sum^{R}_{j=1} \prod^{j-1}_{i=0}  \left(  \frac{\Lambda-i\,K}{2\,\mathcal{M}} \right)  \right).
\end{equation*}
\end{cor}

Corollary \ref{cor:Bishop} can be viewed as an analogue of Bishop (or rather Heintze-Karcher) volume comparison theorem on Riemannian manifold under a lower Ricci curvature bound.

In the undirected case,
there is a further work on volume comparisons (see \cite{BRT}).

\subsection{Laplacian comparisons}\label{sec:Laplacian comparisons}
We are now in a position to give a proof of Theorem \ref{thm:Laplacian comparison}.
\begin{proof}[Proof of Theorem \ref{thm:Laplacian comparison}]
Let $x\in V$.
For $K\in \mathbb{R}$
we assume $\inf_{y\in V\setminus \{x\}}\kappa(x,y) \geq K$.
For $\Lambda \in (-\infty,-1]$
we further assume $\mathcal{H}_{x}\geq \Lambda$.
The distance function $\rho_{x}$ satisfies $\nabla_{xy}\rho_{x}=1$ for the gradient operator $\nabla_{xy}$ defined as (\ref{eq:gradient operator}).
Therefore,
due to Theorem \ref{thm:limit free formula},
\begin{equation*}
K \leq \kappa(x,y) \leq \nabla_{xy} \mathcal{L} \rho_{x} = \frac{\mathcal{L} \rho_{x}(y)- \mathcal{L} \rho_{x}(x)}{d(x,y)}\leq \frac{\mathcal{L} \rho_{x}(y)- \Lambda}{d(x,y)}
\end{equation*}
for every $y\in V\setminus \{x\}$,
and hence (\ref{eq:Laplacian comparison}).
Thus we complete the proof of Theorem \ref{thm:Laplacian comparison}.
\end{proof}

In the rest of this subsection,
we compare Theorem \ref{thm:Laplacian comparison} with a similar result in smooth setting.
As already mentioned in Subsection \ref{sec:Main results and organization},
on manifolds with a lower Ricci curvature bound,
it is well-known that
several comparison geometric results hold for hypersurfaces with a mean curvature bound.
We now compare Theorem \ref{thm:Laplacian comparison} with a Laplacian comparison on weighted manifolds with boundary under a lower Ricci curvature bound,
and a lower mean curvature bound for the boundary obtained in \cite{S}.
We find a similarity between them.

Let $(M,g,\phi)$ be a weighted Riemannian manifold with boundary with weighted measure
\begin{equation*}
\mathfrak{m}_{\phi}:=e^{-\phi}v_{g},
\end{equation*}
where $v_{g}$ is the Riemannian volume measure.
The \textit{weighted Laplacian} is defined as
\begin{equation*}
\mathcal{L}_{\phi}:=\mathcal{L}_{g}+g(\nabla \phi,\nabla \cdot),
\end{equation*}
here $\nabla$ is the gradient.
The \textit{weighted Ricci curvature} is defined as follows (\cite{BE}, \cite{L}):
\begin{equation*}
\ric_{\phi}:=\ric_{g}+\Hess \phi,
\end{equation*}
where $\ric_{g}$ is the Ricci curvature determined by $g$,
and $\Hess$ is the Hessian.
Let $\ric_{\phi,M}$ be its infimum over the unit tangent bundle.
Let $\inte M$ and $\bm$ stand for the interior and boundary of $M$,
respectively.
Let $\rho_{\bm}:M\to \mathbb{R}$ denote the distance function from $\bm$ defined as $\rho_{\bm}:=d_{g}(\bm,\cdot)$,
which is smooth on $\inte M \setminus \cut \bm$.
Here $\cut \bm$ is the cut locus for the boundary (for its precise definition, see e.g., Subsection 2.3 in \cite{S1}).
For $z\in \bm$,
the \textit{weighted mean curvature} of $\bm$ at $z$ is defined as
\begin{equation*}
\mathcal{H}_{\phi,z}:=\mathcal{H}_{g,z}+g(\nabla \phi,u_{z}),
\end{equation*}
where $\mathcal{H}_{g,z}$ is the (inward) mean curvature induced from $g$,
and $u_{z}$ is the unit inner normal vector on $\bm$ at $z$.
Set $\mathcal{H}_{\phi,\bm}:=\inf_{z\in \bm} \mathcal{H}_{\phi,z}$.

The second author \cite{S} has shown the following Laplacian comparison inequality under a similar lower curvature bound to that of Theorem \ref{thm:Laplacian comparison} (see Lemma 6.1 in \cite{S}):
\begin{lem}[\cite{S}]\label{lem:Laplacian comparison with boundary}
For $K\in \mathbb{R}$
we assume $\ric_{\phi,M}\geq K$.
For $\Lambda \in \mathbb{R}$
we further assume $\mathcal{H}_{\phi,\bm}\geq \Lambda$.
Then on $\inte M \setminus \cut \bm$,
we have
\begin{equation}\label{eq:Laplacian comparison with boundary}
\mathcal{L}_{\phi} \rho_{\bm}\geq K \rho_{\bm}+\Lambda.
\end{equation}
\end{lem}

One can observe that
the form of our Laplacian comparison inequality (\ref{eq:Laplacian comparison}) is same as that of (\ref{eq:Laplacian comparison with boundary}).
The second author \cite{S} derived a relative volume comparison of Heintze-Karcher type from Lemma \ref{lem:Laplacian comparison with boundary} (see Theorem 6.3 in \cite{S}, and cf. \cite{HK}, Theorem 2 in \cite{Mo}).

\section{Dirichlet eigenvalues of $p$-Laplacian}\label{sec:Dirichlet eigenvalues of p-Laplacian}

Let $\mathcal{V}$ denote a non-empty subset of $V$ with $\mathcal{V} \neq V$.
The purpose of this last section is to establish a lower bound of the Dirichlet eigenvalues of the $p$-Laplacian on $\mathcal{V}$ under our lower curvature bounds as an application of the study in Subsection \ref{sec:Laplacian comparisons}.

\subsection{Dirichlet $p$-Poincar\'e constants}\label{sec:Dirichlet p-Poincare constants}
Let $p\in (1,\infty)$.
For a non-zero function $f:V\to \mathbb{R}$,
its \textit{$p$-Rayleigh quotient} is defined by
\begin{equation*}
\mathcal{R}_{p}(f):=\frac{1}{2}\frac{\sum_{x,y\in V} \vert f(y)-f(x) \vert^{p}\, \mathfrak{m}_{xy}}{\sum_{x\in V}\vert f(x)\vert^{p}\mathfrak{m}(x)}.
\end{equation*}
We define the \textit{Dirichlet $p$-Poincar\'e constant over $\mathcal{V}$} by
\begin{equation*}
\lambda^{D}_{p}(\mathcal{V}):=\inf_{f\in \mathcal{F}_{\mathcal{V}}\setminus \{0\}} \mathcal{R}_{p}(f),
\end{equation*}
where $\mathcal{F}_{\mathcal{V}}$ denotes the set of all function $f:V\to \mathbb{R}$ with $f|_{V \setminus \mathcal{V}}=0$.

We briefly mention the relation between the Dirichlet $p$-Poincar\'e constant and the Dirichlet eigenvalues of $p$-Laplacian (cf. \cite{GHJ}, \cite{HW}).
The \textit{$p$-Laplacian $\mathcal{L}_{p}:\mathcal{F}\to \mathcal{F}$} is defined by
\begin{equation*}
\mathcal{L}_{p}f(x):=\sum_{y\in V} \vert f(x)-f(y) \vert^{p-2} (f(x)-f(y))\mathcal{P}(x,y).
\end{equation*}
The $2$-Laplacian $\mathcal{L}_{2}$ coincides with the Chung Laplacian $\mathcal{L}$.
A real number $\lambda$ is said to be a \textit{Dirichlet eigenvalue} of $\mathcal{L}_{p}$ on $\mathcal{V}$ if
there is a non-zero function $f\in \mathcal{F}_{\mathcal{V}}$ such that
\begin{equation*}
\mathcal{L}_{p}f=\lambda \vert f\vert^{p-2} \, f.
\end{equation*}
The smallest Dirichlet eigenvalue of the $p$-Laplacian $\mathcal{L}_{p}$ on $\mathcal{V}$ can be variationally characterized as $\lambda^{D}_{p}(\mathcal{V})$.

\subsection{Cheeger inequalities}\label{sec:Cheeger inequalities}
We first formulate an inequality of Cheeger type in our setting to derive a lower bound of the Dirichlet $p$-Poincar\'e constant.
We will refer to the argument of the proof of Theorem 4.8 in \cite{G}, and Theorem 3.5 in \cite{KM}.
We introduce the Dirichlet isoperimetric constant for $\mathcal{V}$.
For a non-empty $\Omega \subset V$,
its \textit{boundary measure} is defined as 
\begin{equation*}
\mathfrak{m}(\partial \Omega):=\sum_{y\in \Omega}\sum_{z\in V\setminus \Omega} \mathfrak{m}_{yz}.
\end{equation*}
We define the \textit{Dirichlet isoperimetric constant on $\mathcal{V}$} by
\begin{equation*}
\mathcal{I}^{D}_{\mathcal{V}}:=\inf_{\Omega}\frac{\mathfrak{m}(\partial \Omega)}{\mathfrak{m}(\Omega)},
\end{equation*}
where $\mathfrak{m}(\Omega)$ is defined as (\ref{eq:measure}),
and the infimum is taken over all non-empty subsets $\Omega \subset \mathcal{V}$.

For $f:V\to \mathbb{R}$ and $t\in \mathbb{R}$,
we set
\begin{equation*}
\Omega_{f,t}:=\{x\in V \mid f(x)> t\}.
\end{equation*}
We present the following co-area formula (cf. Lemma 3.4 in \cite{G}):
\begin{lem}\label{lem:co-area formula}
For every $f:V\to \mathbb{R}$
we have
\begin{equation*}
\int^{\infty}_{-\infty}\,\mathfrak{m}(\partial \Omega_{f,t})\,dt=\frac{1}{2}\sum_{y,z\in V} \vert f(y)-f(z) \vert\, \mathfrak{m}_{yz}.
\end{equation*}
\end{lem}
\begin{proof}
For an interval $I\subset \mathbb{R}$,
let $\mathbf{1}_{I}$ denote its indicator function.
For each $t\in \mathbb{R}$
we see that $\mathfrak{m}(\partial \Omega_{f,t})$ is equal to
\begin{equation*}
\sum_{y\in \Omega_{f,t}} \sum_{z\in V\setminus \Omega_{f,t}}\mathfrak{m}_{yz}=\sum_{f(z)<f(y)}\,\mathbf{1}_{[f(z),f(y))}(t)\,\mathfrak{m}_{yz},
\end{equation*}
where the summation in the right hand is taken over all ordered pairs $(y,z)\in V\times V$ with $f(z)<f(y)$.
Integrating the above equality with respect to $t$ over $(-\infty,\infty)$,
we deduce
\begin{align*}
\int^{\infty}_{-\infty}\,\mathfrak{m}(\partial \Omega_{f,t})\,dt
&=\sum_{f(z)<f(y)} \int^{\infty}_{-\infty} \mathbf{1}_{[f(z),f(y))}(t)\mathfrak{m}_{yz}\\
&=\sum_{f(z)<f(y)}( f(y)-f(z))\, \mathfrak{m}_{yz}=\frac{1}{2}\sum_{y,z\in V} \vert f(y)-f(z) \vert\,\mathfrak{m}_{yz}.
\end{align*}
We obtain the desired equality.
\end{proof}

Lemma \ref{lem:co-area formula} leads us to the following (cf. Lemma 4.9 in \cite{G}):
\begin{lem}\label{lem:implication of co-area formula}
For every non-negative function $f\in \mathcal{F}_{\mathcal{V}}$,
\begin{equation*}
\frac{1}{2}\sum_{y,z\in V} \vert f(y)-f(z) \vert \mathfrak{m}_{yz}\geq \mathcal{I}^{D}_{\mathcal{V}} \sum_{x\in V} f(x)\mathfrak{m}(x).
\end{equation*}
\end{lem}
\begin{proof}
Since $f\in \mathcal{F}_{\mathcal{V}}$,
the set $\Omega_{f,t}$ is contained in $\mathcal{V}$ for every $t\geq 0$,
and hence $\mathfrak{m}(\partial \Omega_{f,t})\geq \mathcal{I}^{D}_{\mathcal{V}}\,\mathfrak{m}(\Omega_{f,t})$.
Lemma \ref{lem:co-area formula} implies
\begin{align*}
\frac{1}{2}\sum_{y,z\in V} \vert f(y)-f(x) \vert \mathfrak{m}_{yz}
&= \int^{\infty}_{0}\,\mathfrak{m}(\partial \Omega_{f,t})\,dt\geq \mathcal{I}^{D}_{\mathcal{V}}\,\int^{\infty}_{0}\,\mathfrak{m}(\Omega_{f,t})\,dt\\
&=\mathcal{I}^{D}_{\mathcal{V}}\,\sum_{x\in V}\mathfrak{m}(x)\int^{\infty}_{0}\mathbf{1}_{[0,f(x))}(t)\,dt 
=\mathcal{I}^{D}_{\mathcal{V}}\,\sum_{x\in V}f(x)\mathfrak{m}(x).
\end{align*}
This proves the lemma.
\end{proof}

We recall the following inequality that has been obtained by Amghibech \cite{Am} (see \cite{Am}, and see also Lemma 3.8 in \cite{KM}):
\begin{lem}[\cite{Am}, \cite{KM}]\label{lem:p-convexity}
Let $p\in (1,\infty)$.
Then for every non-negative function $f:V\to \mathbb{R}$,
and for all $x,y\in V$
we have
\begin{equation*}
\vert f(y)^{p}-f(x)^{p} \vert \leq p\,\vert f(y)-f(x) \vert\, \left(  \frac{f(y)^{p}+f(x)^{p}}{2}  \right)^{1/q},
\end{equation*}
where $q$ is determined by $p^{-1}+q^{-1}=1$.
\end{lem}

Summarizing the above lemmas,
we conclude the following inequality of Cheeger type (cf. Theorem 4.8 in \cite{G}):
\begin{prop}\label{prop:Dirichlet Cheeger inequality}
For $p\in (1,\infty)$ we have
\begin{equation*}
\lambda^{D}_{p}(\mathcal{V})\geq \frac{2^{p-1}}{p^{p}} (\mathcal{I}^{D}_{\mathcal{V}})^{p}.
\end{equation*}
\end{prop}
\begin{proof}
We fix $f\in \mathcal{F}_{\mathcal{V}}\setminus \{0\}$.
We apply Lemma \ref{lem:implication of co-area formula} to a non-negative function $\vert f \vert^{p}$ which belongs to $\mathcal{F}_{\mathcal{V}}$.
Using Lemma \ref{lem:p-convexity} and the triangle inequality,
we obtain
\begin{align*}
\mathcal{I}^{D}_{\mathcal{V}} \sum_{x\in V} \vert f(x) \vert^{p} \mathfrak{m}(x) 
&\leq  \frac{1}{2}\sum_{x,y\in V} \left\vert \vert f(y)\vert^{p}-\vert f(x) \vert^{p} \right\vert\, \mathfrak{m}_{xy}\\
& \leq \frac{p}{2}\sum_{x,y\in V}  \vert \vert f(y)\vert-\vert f(x) \vert \vert \left(  \frac{\vert f(y) \vert^{p}+\vert f(x)\vert^{p}}{2}  \right)^{1/q}\mathfrak{m}_{xy} \\
& \leq  \frac{p}{2}\sum_{x,y\in V} \vert f(y)- f(x) \vert  \left(  \frac{\vert f(y)\vert^{p}+\vert f(x)\vert^{p}}{2}  \right)^{1/q} \mathfrak{m}_{xy},
\end{align*}
where $q$ is determined by $p^{-1}+q^{-1}=1$.
The H\"older inequality yields
\begin{align*}
\mathcal{I}^{D}_{\mathcal{V}} \sum_{x\in V} \vert f(x) \vert^{p} \mathfrak{m}(x)&\leq  \frac{p}{2}\sum_{x,y\in V} \vert f(y)- f(x) \vert \, \mathfrak{m}^{1/p}_{xy}\,\, \left(  \frac{\vert f(y)\vert^{p}+\vert f(x)\vert^{p}}{2}\, \mathfrak{m}_{xy} \right)^{1/q}\\
&\leq  \frac{p}{2}  \left(  \sum_{x,y\in V} \vert f(y)-f(x) \vert^{p}\, \mathfrak{m}_{xy}\right)^{1/p}  \,\, \left(\frac{1}{2} \sum_{x,y\in V}  \left(  \vert f(y)\vert^{p}+\vert f(x)\vert^{p}  \right)\,\mathfrak{m}_{xy}  \right)^{1/q}\\
&=  \frac{p}{2}  \left(  \sum_{x,y\in V} \vert f(y)-f(x) \vert^{p}\, \mathfrak{m}_{xy}\right)^{1/p}  \,\, \left(\sum_{x\in V}  \vert f(x)\vert^{p} \mathfrak{m}(x)  \right)^{1/q}.
\end{align*}
We possess
\begin{equation*}
(\mathcal{I}^{D}_{\mathcal{V}})^{p} \leq \left(\frac{p}{2}\right)^{p}  \frac{\sum_{x,y\in V} \vert f(y)-f(x) \vert^{p}\, \mathfrak{m}_{xy}}{\sum_{x\in V} \vert f(x) \vert^{p} \mathfrak{m}(x)}=\frac{p^{p}}{2^{p-1}}\mathcal{R}_{p}(f).
\end{equation*}
Thus we arrive at the desired inequality.
\end{proof}

\begin{rem}
We provide a brief historical remark on the Cheeger inequality (without boundary condition) for graphs (for more details, cf. \cite{TH} and the references therein).
Alon-Milman \cite{AM}, Alon \cite{A} established the Cheeger inequality for undirected graphs, and for the graph Laplacian.
Chung \cite{C1} extended it to the directed case.
Amghibech \cite{Am} generalized it for the graph $p$-Laplacian in the undirected case.
\end{rem}

\subsection{Dirichlet eigenvalue estimates}\label{sec:Dirichlet eigenvalue estimates}
For $x\in V$ and $R \geq 1$
we set
\begin{equation*}
E_{R}(x):=\{y\in V \mid \rho_{x}(y)\geq R\}.
\end{equation*}
We obtain the following isoperimetric inequality for $E_{R}(x)$:
\begin{prop}\label{prop:isoperimetric inequality}
Let $x\in V$.
For $K\in \mathbb{R}$
we assume $\inf_{y\in V \setminus \{x\}}\kappa(x,y) \geq K$.
For $\Lambda \in (-\infty,-1]$
we also assume $\mathcal{H}_{x}\geq \Lambda$.
For $D>0$
we further assume $\IR_{x} V\leq D$.
Then for every $R \geq 1$ with $KR+\Lambda>0$,
we have
\begin{equation*}
\mathcal{I}^{D}_{E_{R}(x)}\geq \frac{K R+\Lambda}{D}.
\end{equation*}
\end{prop}
\begin{proof}
Fix a non-empty $\Omega \subset E_{R}(x)$.
From Proposition \ref{prop:integration by parts} we derive
\begin{align*}
-\sum_{y\in \Omega}\mathcal{L} \rho_{x}(y)\mathfrak{m}(y)&=\sum_{y\in \Omega} \sum_{z\in V\setminus \Omega}(\rho_{x}(z)-\rho_{x}(y))\mathfrak{m}_{yz}\\
                                                                               &\geq -\sum_{y\in \Omega} \sum_{z\in V\setminus \Omega}\rho_{x}(y)\mathfrak{m}_{yz}\geq -D\,\mathfrak{m}(\partial \Omega).
\end{align*}
On the other hand,
Theorem \ref{thm:Laplacian comparison} leads us that
for all $y\in \Omega$
\begin{equation*}
\mathcal{L} \rho_{x}(y) \geq K \rho_{x}(y)+\Lambda\geq KR+\Lambda.
\end{equation*}
Therefore,
\begin{equation*}
(K R+\Lambda)\mathfrak{m}(\Omega)\leq \sum_{y\in \Omega}\mathcal{L} \rho_{x}(y)\mathfrak{m}(y)\leq D\,\mathfrak{m}(\partial \Omega).
\end{equation*}
We complete the proof.
\end{proof}

\begin{rem}\label{rem:isop Riem graph}
On Riemannian manifolds with boundary with a lower Ricci curvature bound and a lower mean curvature bound for the boundary,
it is well-known that
one can derive a lower bound of its Dirichlet isoperimetric constant from a Laplacian comparison theorem for the distance function from the boundary, and integration by parts formula (see Proposition 4.1 in \cite{K}, Lemma 8.9 in \cite{S1}, and cf. Theorem 15.3.5 in \cite{Sh}).
Proposition \ref{prop:isoperimetric inequality} can be viewed as an analogue of such a result on manifolds with boundary (cf. Subsection \ref{sec:Laplacian comparisons}).
\end{rem}

We are now in a position to prove Theorem \ref{thm:main theorem}.
\begin{proof}[Proof of Theorem \ref{thm:main theorem}]
Let $x\in V$ and $p\in (1,\infty)$.
For $K\in \mathbb{R}$
we assume $\inf_{y\in V \setminus \{x\}}\kappa(x,y) \geq K$.
For $\Lambda \in (-\infty,-1]$
we also assume $\mathcal{H}_{x}\geq \Lambda$.
For $D>0$
we further assume $\IR_{x} V\leq D$.
Combining Propositions \ref{prop:Dirichlet Cheeger inequality} and \ref{prop:isoperimetric inequality},
we have
\begin{equation*}
\lambda^{D}_{p}(E_{R}(x))\geq \frac{2^{p-1}}{p^{p}} (\mathcal{I}^{D}_{E_{R}(x)})^{p} \geq \frac{2^{p-1}}{p^{p}} \left(\frac{K R+\Lambda}{D}\right)^{p}.
\end{equation*}
We arrive at the desired inequality (\ref{eq:main theorem}).
Thus we complete the proof of Theorem \ref{thm:main theorem}.
\end{proof}


\end{document}